\documentclass[12pt]{amsart}
\usepackage{amscd,amsmath,amsthm,amssymb,graphics, comment,latexsym }
\usepackage[left]{lineno}
\usepackage{color}
\definecolor{mypink}{rgb}{1.0, 0.01, 0.24}
\definecolor{mygreen}{rgb}{0, 0.7, 0.30}
 \def\NZQ{\mathbb}               
 \def\NN{{\NZQ N}}
 
 \def\ZZ{{\NZQ Z}}
 \def\RR{{\NZQ R}}

 \def\frk{\mathfrak}               

 \def\mm{{\frk m}}

 \def\G{{\mathcal G}}

  \def\Hc{{\mathcal H}}
  \def\Nc{{\mathcal N}}

 \def\ab{{\mathbf a}}
 \def\bb{{\mathbf b}}
 \def\xb{{\mathbf x}}
 
 \def\pb{{\mathbf p}}
\def\qb{{\mathbf q}}
 \def\zb{{\mathbf z}}
 \def\cb{{\mathbf c}}

\def\hb{{\mathbf h}}
 \def\ub{{\mathbf u}}
 \def\vb{{\mathbf v}}

 \def\nb{{\mathbf n}}
\def\hb{{\mathbf h}}
 \def\opn#1#2{\def#1{\operatorname{#2}}} 
 \opn\chara{char} \opn\length{\lambda} \opn\pd{pd} \opn\rk{rk}
 \opn\projdim{proj\,dim} \opn\injdim{inj\,dim} \opn\rank{rank}
 \opn\depth{depth} \opn\grade{grade} \opn\height{height}
 \opn\embdim{emb\,dim} \opn\codim{codim}
 
 \opn\Tr{Tr} \opn\bigrank{big\,rank}
 \opn\superheight{superheight}\opn\lcm{lcm}
 \opn\trdeg{tr\,deg}
 \opn\reg{reg} \opn\lreg{lreg} \opn\ini{in} \opn\lpd{lpd}
 \opn\size{size} \opn\sdepth{sdepth}
 \opn\link{link}\opn\fdepth{fdepth}\opn\lex{lex}
 \opn\tr{tr}
 \opn\type{type}
 \opn\gap{gap}
 \opn\arithdeg{arith-deg}
 \opn\astab{astab}
  \opn\dstab{dstab}
 \opn\defect{def}
 \opn\div{div} \opn\Div{Div} \opn\cl{cl} \opn\Cl{Cl}
 \opn\Spec{Spec} \opn\Supp{Supp} \opn\supp{supp} \opn\Sing{Sing}
 \opn\Ass{Ass} \opn\Min{Min}\opn\Mon{Mon}
 \opn\Ann{Ann} \opn\Rad{Rad} \opn\Soc{Soc}
 \opn\Im{Im} \opn\Ker{Ker} \opn\Coker{Coker} \opn\Am{Am}
 \opn\Hom{Hom} \opn\Tor{Tor} \opn\Ext{Ext} \opn\End{End}
 \opn\Aut{Aut} \opn\id{id}
 
 \opn\nat{nat}
 \opn\pff{pf}
 \opn\Pf{Pf} \opn\GL{GL} \opn\SL{SL} \opn\mod{mod} \opn\ord{ord}
 \opn\Gin{Gin} \opn\Hilb{Hilb}\opn\sort{sort}
 \opn\PF{PF}\opn\Ap{Ap}
 \opn\mult{mult}
 \opn\aff{aff}
 \opn\relint{relint} \opn\st{st}
 \opn\lk{lk} \opn\cn{cn} \opn\core{core} \opn\vol{vol}  \opn\inp{inp} \opn\nilpot{nilpot}
 \opn\link{link} \opn\star{star}\opn\lex{lex}\opn\set{set}
 \opn\width{wd}
 \opn\Fr{F}
 \opn\QF{QF}
 \opn\G{G}
 \opn\type{type}\opn\res{res}
 \opn\conv{conv}
 \opn\gr{gr}
 
 \def\pot#1#2{#1[\kern-0.28ex[#2]\kern-0.28ex]}

 \opn\dirlim{\underrightarrow{\lim}}
 \opn\inivlim{\underleftarrow{\lim}}
 \let\union=\cup
 \let\sect=\cap

 \let\to=\rightarrow
 
 \def\Implies{\ifmmode\Longrightarrow \else
         \unskip${}\Longrightarrow{}$\ignorespaces\fi}
 \def\implies{\ifmmode\Rightarrow \else
         \unskip${}\Rightarrow{}$\ignorespaces\fi}
 \def\iff{\ifmmode\Longleftrightarrow \else
         \unskip${}\Longleftrightarrow{}$\ignorespaces\fi}

 \let\:=\colon
 \theoremstyle{plain}
 \newtheorem{Theorem}{Theorem}[section]
 \newtheorem{Lemma}[Theorem]{Lemma}
 \newtheorem{Corollary}[Theorem]{Corollary}
 \newtheorem{Proposition}[Theorem]{Proposition}

 \newtheorem{Definition}[Theorem]{Definition}
  \newtheorem{Example}[Theorem]{Example}
 \newtheorem{Examples}[Theorem]{Examples}

\newtheorem{Remark}[Theorem]{Remark}

\newtheorem{Notation}[Theorem]{Notation}

 \let\epsilon\varepsilon
 \let\kappa=\varkappa
 \def\qed{\ifhmode\textqed\fi
       \ifmmode\ifinner\quad\qedsymbol\else\dispqed\fi\fi}
 \def\textqed{\unskip\nobreak\penalty50
        \hskip2em\hbox{}\nobreak\hfil\qedsymbol
        \parfillskip=0pt \finalhyphendemerits=0}
 \def\dispqed{\rlap{\qquad\qedsymbol}}
 
 \opn\dis{dis}
 \def\pnt{{\raise0.5mm\hbox{\large\bf.}}}
 
 \opn\Lex{Lex}

\begin{document}

\title{Ulrich elements in normal simplicial affine  semigroups}

\author {J\"urgen Herzog,  Raheleh Jafari, Dumitru I.\ Stamate}

\address{J\"urgen Herzog, Fachbereich Mathematik, Universit\"at Duisburg-Essen, Campus Essen, 45117
Essen, Germany} \email{juergen.herzog@uni-essen.de}

\address{Raheleh Jafari,  Mosaheb Institute of Mathematics, Kharazmi University, and School of Mathematics, Institute for Research in Fundamental Sciences (IPM), P.O. Box 19395-5746, Tehran, Iran.} \email{rjafari@ipm.ir}

\address{Dumitru I. Stamate, Faculty of Mathematics and Computer Science, University of Bucharest, Str. Academiei 14, Bucharest -- 010014, Romania }
\email{dumitru.stamate@fmi.unibuc.ro}



\dedicatory{ }

\begin{abstract}
Let $H\subseteq \NN^d$ be a normal 
affine semigroup, $R=K[H]$ its semigroup ring over the field $K$ and $\omega_R$ its canonical module.
The Ulrich elements for $H$ are those $h$ in $H$  such that for the multiplication  map by $\xb^h$ from $R$ into $\omega_R$, the cokernel is an Ulrich module. We say that the ring $R$ is almost Gorenstein if Ulrich elements exist in $H$.
For the class of slim semigroups that we introduce, we provide an algebraic criterion for testing the Ulrich propery.
When $d=2$, all normal affine semigroups are slim. Here we have a simpler combinatorial description of the Ulrich property.
We improve this result for testing the elements in $H$ which are closest to zero. In particular, we give a simple arithmetic criterion for when is $(1,1)$ an Ulrich element in $H$.
\end{abstract}

\thanks{}

\maketitle

\section*{Introduction}

Let $H$ be an affine semigroup in $\NN^d$ and $K[H]$ its semigroup ring over the field $K$.
In this paper we investigate   the almost Gorenstein property for $K[H]$ taking into account the natural
multigraded structure of this ring, under the assumption that $H$ is  normal and simplicial.

The almost Gorenstein property appeared  in the work of Barucci and Fr\"oberg \cite{BF} in the context of 1-dimensional analytical unramified rings. It  was extended to 
$1$-dimensional local rings by Goto, Matsuoka and Thi Phuong in \cite{GMT}, and later on to rings of higher dimension by Goto, Takahashi and Taniguchi in \cite{GTT-2015}.
Let $R$ be a positively graded Cohen-Macaulay $K$-algebra with canonical module $\omega_R$. We let $a=-\min\{k\in \ZZ: (\omega_R)_k \neq 0 \}$, which is also known as the $a$-invariant of $R$. In \cite{GTT-2015}, $R$ is called (graded) almost Gorenstein (AG for short) if there exists an exact sequence of graded $R$-modules
\begin{equation}
\label{ses}
0 \to R \to \omega_R(-a) \to E \to 0,
\end{equation}
where $E$ is an Ulrich module, i.e. $E$ is a Cohen-Macaulay graded module which is minimally generated by $e(E)$ elements. Here $e(E)$ denotes the multiplicity of $E$ with respect to   the graded maximal ideal in $R$.

Let $H \subseteq  \NN^d$ be an affine semigroup  whose associated group is  $\gr(H) =\ZZ^d$.
  We denote $C$ the cone over $H$.
Assume $H$ is normal, i.e. $H=C\sect \ZZ^d$, equivalently, the ring $R$ is normal. Then $R$ is a Cohen-Macaulay ring (\cite{H}) and a $K$-basis for the canonical module $\omega_R$ is given by the monomials with exponents in the relative interior of the cone $C$ (\cite{Danilov}, \cite{Stanley}), i.e.  in the set $\omega_H= \ZZ^d \sect \relint C$. In the multigraded setting that we want to consider here, there does not seem to be any distinguished  element in $\omega_H$ to replace the $a$-invariant in the short exact sequence \eqref{ses}.
In this sense, we propose the following.

\medskip

\noindent {\bf Definition \ref{def:ag}.}
 For $\bb \in\omega_H$ consider the following exact sequence
\begin{eqnarray}
0\to R \to  \omega_R(\bb) \to  E \to 0,
\end{eqnarray}
where $1\in R$ is mapped to $u=\xb^\bb$ and  $E =\omega_R/uR$.
Then  $\bb$ is called  an \textit{Ulrich element} in $H$, if $E$ is an Ulrich $R$-module.

If $H$ admits  an Ulrich element $\bb$, then we call the ring $R=K[H]$ {\em almost Gorenstein with respect to $\bb$}, or  simply  AG if $H$ has an Ulrich element.

\medskip

The Gorenstein property has attracted a lot of interest due to its multifaceted  algebraic and homological descriptions. For rings with a combinatorial structure behind, there are often nice characterizations of the Gorenstein property. Scratching only at the surface, we mention that Gorenstein toric rings were characterized by Hibi in \cite{Hi}, and for special subclasses the results are more precise, see \cite{DeNegri-Hibi, TH87, HLMMV, Dinu}.

The almost Gorenstein property was characterized for determinantal rings \\(Taniguchi, \cite{Tan}), numerical semigroup rings 
(Nari, \cite{Nari-2013}) and  Hibi rings (Miyazaki, \cite{Miyazaki}).

In this paper we investigate the
 Ulrich elements in $H$ under the assumption that the normal affine semigroup $H\subset \NN^d$ is also simplicial, i.e. the cone $C$ over $H$ has $d=\dim_\RR \aff(H)$ extremal rays. 
That will be assumed for the rest of this introduction, too.

Next we outline the main results.
We denote $\ab_1,\dots, \ab_d$ the primitive integer vectors in $H$ situated on each extremal ray of the cone $C$, respectively, and we call them the extremal rays of $H$. They are part of the Hilbert basis of $H$, denoted $B_H$, which is the unique minimal generating set of $H$. 

When $H$ is normal and simplicial it is known that the monomials $\xb^{\ab_1}, \dots, \xb^{\ab_d}$  form a maximal regular sequence on $R$.
A first result that we prove in Proposition~\ref{regular} is that for any $\bb\neq 0$ in $H$ the sequence $\xb^\bb, \xb^{\ab_1}-\xb^{\ab_2},\dots, \xb^{\ab_1}-\xb^{\ab_d}$ is regular, as well. Let $J=(\xb^{\ab_i}-\xb^{\ab_j}: 1\leq i, j \leq d)R$. 
The next technical result is vital in our study of Ulrich elements.  Namely, in Theorem~\ref{tricky}, we show that  $J$ is a reduction ideal of $\mm$ modulo the ideal $\xb^\bb R$, if and only 
if  for any $\cb \in B_H\setminus \omega_H$ the sum of the coordinates of $\cb$ with respect to the basis $\ab_1,\dots, \ab_d$ is at least one.
The normal and simplicial  semigroups with that property are to be called {\em slim}.

In this notation, we  provide the following   characterization. 

\medskip
\noindent{\bf Theorem~\ref{criterion}.}  Let  $H$ be a slim semigroup  and $\bb \in \omega_H$.  Then $\bb$  is an Ulrich element  in $H$ if and only if  $\mm  \omega_R \subseteq (\xb^\bb R, J \omega_R)$.

\medskip
 
This result allows to produce  first examples of semigroups with Ulrich elements, see Examples~\ref{ex:nice},~\ref{gore-regular}. 

In the next sections we focus on making more explicit the AG property in dimension two.
Let $H$ be any normal  affine semigroup $H\subseteq \NN^2$.  It is automatically slim since it is simplicial and  $B_H\setminus\{\ab_1,\ab_2\} \subseteq \omega_H$ (see Lemma~\ref{lemma:omega}).
 We denote $\ab_1, \ab_2$ its extremal rays. In Theorem~\ref{two} we prove that any element  $\bb \in B_H\setminus \{\ab_1, \ab_2\}$ is an Ulrich element in $H$ if and only if for all $\cb_1,\cb_2$ in $ B_H$ one has $$\cb_1+\cb_2 \in (\ab_1+H)\cup (\ab_2+H)\cup (\bb+H).$$ Equivalently, if for all $\cb_1,\cb_2 \in B_H$ so that $\cb_1, \cb_2 \in P_H$ it follows that $\cb_1+\cb_2 \in \bb+H$.  Here $P_H=\{ \lambda_1 \ab_1+\lambda_2 \ab_2: 0\leq \lambda_1, \lambda_2 <1\}$ is the fundamental parallelogram of $H$.

Based on this result, in Section \ref{section:dim2} we find examples  with zero, one, or several Ulrich elements  in $B_H$.
 
We prove in Lemma~\ref{exists} that for any $H\subseteq \NN^2$ as above, the semigroup ideal $\omega_H$ has a unique minimal element with respect to the componentwise partial order on $\NN^2$. We call it the bottom element of $H$.
This definition naturally extends to higher embedding dimension, but when $d>2$ not all normal semigroups in $\NN^d$   have a bottom element.  

However, bottom elements, when available, are good candidates to check against the Ulrich property.
We prove that when $H\subseteq \NN^d$ is a slim semigroup such that
\begin{itemize}
\item (Proposition~\ref{1}) the nonzero elements in $H$ have all the entries positive and $\bb=(1,1,\dots, 1) \in \omega_H$, or
\item (Proposition~\ref{twoinp}) $d=2$ and $\bb$ the bottom element in $H$ satisfies $2\bb \in P_H$,
\end{itemize}
then $\bb$ is the only possible  Ulrich element in $H$.

These results motivate us to find more direct criteria  for testing the Ulrich property of the bottom element. Our attempt is successful  when $d=2$. 

\medskip

In the following, $H$ is a normal affine semigroup in $\NN^2$ with the extremal rays $\ab_1=(x_1, y_1)$ and $\ab_2=(x_2, y_2)$ with $\ab_1$ closer to the $x$-axis than $\ab_2$. Considering $\bb=(u,v)$ the bottom element  of $H$, for $i=1,2$ we define $H_i$ to be the normal semigroup with the extremal rays $\bb$ and $\ab_i$. We denote $H_i^*=\relint P_{H_i} \cap \ZZ^2$ for $i=1,2$. We show

\medskip

\noindent{\bf {Lemma~\ref{AG2}.}} For $\bb$ the bottom element in $H$ the following are equivalent:
\begin{enumerate}
\item [(a)] $\bb$ is an Ulrich element in $H$;
\item [(b)] for $i=1,2$, if $\pb,\qb \in H_i^*$ then $\pb+\qb \notin H_i^*$.
\end{enumerate}

 We shall say that $H$ is AG1 if point (b) above is satisfied for $i=1$ and we call it AG2 if it holds for $i=2$.

Thus the bottom element is an Ulrich element in $H$ if and only if $H$ is AG1 and AG2. This calls for a better understanding of the points in $H_1^*$ and $H_2^*$. Lemma~\ref{number} shows that $H_i^*$ has  $|vx_i-uy_i|-1$ elements, for $i=1,2$. An immediate consequence of independent interest is the following Gorenstein criterion.

\medskip

\noindent {\bf Corollary~\ref{gor}.}  With notation as above, the ring $K[H]$ is Gorenstein if and only if $vx_1-uy_1=uy_2-vx_2=1$.

The $x$-coordinates of points in $H_1^*$  are distinct integers in the interval $(u, x_1)$. Moreover, if for any integer $i$ we consider the integers $q_i, r_i$ so that $iy_1=q_i x_1+r_i$ with $0\leq r_i <x_1$ then any integer $k\in (u, x_1)$ is the $x$-coordinate of some $\pb \in H_1^*$ (i.e. $k\in \pi_1(H_1^*)$) if and only if $q_k=v-1+q_{k-u}$, or equivalently, if $r_k \geq x_1 -(vx_1-uy_1)$. In that case, $\pb=(k, q_k+1)$. These observations  (detailed in  Lemma~\ref{k-point}) allow us to test the AG1 property as follows.

\medskip

\noindent {\bf Proposition~\ref{AG1-rk}.} The semigroup $H$ is AG1 if and only if $r_k+r_\ell < 2x_1-(vx_1-uy_1)$ for all integers $k,\ell \in \pi_1(H_1^*)$ with $k+\ell < x_1$.

When the bottom element is $(1,1)$ (i.e. $y_1<x_1$ and $x_2<y_2$) we can describe recursively the points in $H_1^*$.

\medskip

\noindent {\bf Lemma~\ref{elements-H1}} Assume $(1,1)\in \omega_H$ and $H_1^* \neq \emptyset$. Let $n=|H_1^*|=x_1-y_1-1$.  
Recursively, we define non-negative integers $\ell_1,\ldots, \ell_n$ and $s_1,\ldots,s_n$ by
\[
x_1=\ell_1(x_1-y_1)+s_1,   \quad \text{with} \quad s_1<x_1-y_1,
\]
and
\[
y_1+s_{i-1}=\ell_{i}(x_1-y_1)+s_{i}\quad \text{with} \quad  s_i<x_1-y_1,
\]
for $i=2,\ldots,n$. Then
	\[
	H_1^*=\left\{\pb_t=(c_t,d_t) : c_t=t+\sum^t_{i=1}\ell_i \ , \ d_t=\sum^t_{i=1}\ell_i \ , \ t=1,\dots,n\right\}.
	\]
	 
A similar description is available for points in $H_2^*$. A little bit more effort is necessary to obtain the following arithmetic  criterion for the Ulrich property of $(1,1)$. The effort is compensated with the simplicity of the statement. 

\medskip

\noindent {\bf{Theorem~\ref{ulrich11}.}} Assume $(1,1)\in \omega_H$. Then $(1,1)$ is an Ulrich element in $H$ if and only if $x_i \equiv 1 \mod (x_i-y_i)$ for $i=1,2$.

Consequently, by Corollary~\ref{ag11}, if $x_1y_1x_2y_2 \neq 0$ the ring $K[H]$ is AG if and only if $x_i \equiv 1 \mod (x_i-y_i)$ for $i=1,2$.

In Section~\ref{sec:ng} we discuss another extension of the Gorenstein property for affine semigroup rings.
According  to the definition proposed in \cite{HHS},  any Cohen-Macaulay ring $K[H]$ is called nearly Gorenstein if
the trace ideal $\tr(\omega_{K[H]})$ contains the graded maximal ideal of $K[H]$. 
For one dimensional rings, the almost Gorenstein property implies the nearly Gorenstein property, but for rings of  larger dimension there
 is no implication between these two properties.
We prove in Theorem~\ref{nearly} that when $H$ is a normal semigroup in $\NN^2$ the ring $K[H]$ is nearly Gorenstein. Example~\ref{example2} shows that the statement is not valid in higher embedding dimensions.

\section{Background on affine semigroups and their toric rings}
\label{background}

  In this paper all semigroups considered are fully embedded, i.e. when writing $H\subseteq \NN^d$ we shall implicitly assume that  the group generated by $H$ is $\gr (H)=\ZZ^d$.
A subset $H\subseteq \NN^d$ is called an affine semigroup if there exist $\cb_1,\dots, \cb_r \in H$ such that $H=\sum_{i=1}^r \NN \cb_i$. Moreover, $H$ is called a normal semigroup if for all $\hb$ in $\NN^d$ and $n$ positive integer,  $n \hb \in H$  implies that $\hb \in H$.

Let $K$ be any field and $H=\sum_{i=1}^r \NN \cb_i \subseteq \NN^d$. The semigroup ring $K[H]$ is the subalgebra of the polynomial ring $K[x_1,\dots, x_d]$ generated by the monomials with exponents in $H$. Then $H$ is normal if and only if the semigroup ring $K[H]$ is integrally closed  in its field of fractions (\cite{BG}).
 The normality for $H$ is also equivalent to the fact  that $H$ contains all the lattice points of the rational polyhedral cone $C$ that it generates, i.e. $H=C\sect \ZZ^d$, where
$$
C=\left \{ \sum_{i=1}^r \lambda_i \cb_i: \lambda_i \in \RR_{\geq 0}, \text{ for } i=1,\dots, r \right\}.
$$

The dimension (or rank) of $H$ is defined as the dimension of $\aff(H)$, the affine subspace it generates. The latter is the same as $\aff(C)$.
 
Let $\langle \cdot, \cdot \rangle$ denote the usual scalar product in $\RR^d$.
Given $\nb \in \RR^d\setminus \{0\}$, the hyperplane $H_\nb=\{\zb \in \RR^d: \langle \zb, \nb \rangle =0\}$ is called a  \textit{support hyperplane} for $C$ if $\langle \zb, \nb \rangle \geq 0$ for all $\zb \in C$  and    $H_\nb \sect C \neq \emptyset$.
In this case, the cone $H_\nb \sect C$ is called a  \textit{face} of $C$ and its dimension is $\dim \aff(H_\nb \sect C)$. Let $F$ be any face of the cone $C$. When $\dim F=1$, the face $F$ is called an  \textit{extremal ray}, and when $\dim F=d-1$, it is called a  \textit{facet} of $C$.  The normal vector to any hyperplane is determined up to multiplication by a nonzero factor, hence we may choose  $\nb_1, \dots, \nb_s \in \ZZ^d$ to be  the normals to the support hyperplanes that determine the facets of  $C$ and such that
$$
C= \{ \zb \in \RR^d: \langle \zb, \nb_i \rangle \geq 0, \text{ for } i=1, \dots, s\}.
$$

The unique minimal set of generators for the semigroup $H$ is called the  \textit{Hilbert basis} of $H$ and we denote it as $B_H$.
 
It is known that the cone $C$ has at least $d$ facets and at least $d$ extremal rays. When $C$ has $d$ facets (equivalently, that it has $d$ extremal rays) the cone $C$ and the semigroup $H$ are called  \textit{simplicial}.

For any $d\geq 2$ we denote by $\Nc_d$ the class of normal simplicial affine semigroups which are fully embedded in $\NN^d$.

Let $H\in\Nc_d$ and $C$ the cone over $H$. On each extremal ray of $C$ there exists a unique primitive element from $H$, which we call an {\em  extremal ray} for $H$. 
Denote $\ab_1,\dots, \ab_d$ the extremal rays for $H$. These form an $\RR$-basis in $\RR^d$. 
 For $\zb \in \RR^d$ such that $\zb=\sum_{i=1}^d \lambda_i \ab_i$ with $\lambda_i \in \RR, i=1,\dots,d$, we set $[\zb]_i=\lambda_i$ for $i=1,\dots, d$.  In this notation, $\zb$ is in the cone $C$  if and only if $[\zb]_i \geq 0 $ for $i=1, \dots, d$. Also, when $\zb \in \ZZ^d$ one has that $\zb \in H$ if and only if $[\zb]_i \geq 0$ for $i=1,\dots, d$.

The fundamental (semi-open) \textit{parallelotope} of $H$ is the set
$$
P_H=\left\{ \zb \in\RR^d: \zb=\sum_{i=1}^d \lambda_i \ab_i \text { with } 0\leq \lambda_i <1 \text{ for } i=1, \dots, d\right\}.
$$
Its closure in $\RR^d$ is the set $ \overline{P_H}=\left\{ \zb \in\RR^d:    0\leq [\zb]_i \leq 1 \text{ for } i=1, \dots, d\right\}$.
It is well known, and easy to see, that any $\hb$ in $H$ decomposes uniquely as $\hb=\sum_{i=1}^d n_i \ab_i+ \hb'$ with $\hb' \in P_H \cap \ZZ^d$ and $n_1,\dots, n_d$  nonnegative integers. 
The extremal rays of $H$ are in $B_H \setminus P_H$, but the rest of the elements in $B_H$ belong to $P_H$.

 \medskip

Since $H$ is simplicial,   $\xb^{\ab_1} ,\dots, \xb^{\ab_d}$  is a system of parameters in $R$, see \cite[(1.11)]{GSW}.
As $H$ is a normal semigroup, by Hochster \cite{H}, the  ring $R=K[H]$ is Cohen-Macaulay of dimension $d$, hence  any system of parameters in $R$   is   a regular sequence  of maximal length.
 By Danilov \cite{Danilov} and Stanley \cite{Stanley}, the canonical  module $\omega_R$ of $R$ is the ideal in $R$ generated by the monomials $\xb^\vb$ whose exponent vector $\vb=\log(\xb^\vb)$  belongs to the relative interior of $C$,  denoted by $\relint C$. Note that
\[
\relint C=\left\{\cb\in \RR^d\: \cb=\sum_{i=1}^d \lambda_i \ab_i \text{ with  $\lambda_i \in \RR_{>0}$ for all $i=1, \dots, d$}\right\}.
\]
We set
$$
\omega_H=\{\hb\in H: \xb^\hb \in \omega_R\}=\ZZ^d\cap\relint C,
$$
which is a semigroup ideal of $H$, i.e.  $\omega_H +H \subseteq \omega_H$.
We note that $\hb \in \ZZ^d$ is in $\omega_H$ if and only if $[\hb]_i >0$ for $i=1,\dots, d$.

 The ideal $\omega_R$ has a unique minimal system of monomial generators which we denote by $G(\omega_R)$. We set $G(\omega_H)=\{ \log(u): u\in G(\omega_R) \}$. Clearly, $G(\omega_H)$ is the unique  minimal system  	of  generators for $\omega_H$.
The situation when $G(\omega_H)$ is a singleton corresponds to the situation when $R$ is a Gorenstein ring.
When $B_H=\{\ab_1,\dots, \ab_d \}$ then $R$ is a regular ring, there is no lattice point in the relative interior of $\overline{P_H}$, and $G(\omega_H)= \{ \sum_{i=1}^d \ab_i\}$.

 The following easy  lemma  describes the minimal generators of $\omega_H$ when $H\in \Nc_2$. 
For completeness,   we  include a proof here.

 \begin{Lemma}
	\label{lemma:omega} 
           Let $H$ in $\Nc_2$ with the extremal rays $\ab_1,\ab_2$.
 Then  $B_H\cap \omega_H = B_H\setminus \{ \ab_1,\ab_2\}$. Moreover, if $\{\ab_1, \ab_2\} \subsetneq B_H$,  then $G(\omega_H)=B_H\cap \omega_H$.
 \end{Lemma}

\begin{proof}
 Clearly, $B_H\sect \omega_H \subseteq B_H\setminus \{\ab_1, \ab_2\}$. If $\bb \in B_H\setminus \{\ab_1, \ab_2\}$ then, since $\ab_1$ and $\ab_2$ are the extremal rays, it follows that $\bb \in \relint P_H$, hence $\bb \in B_H\sect \omega_H$. Therefore, $B_H\cap \omega_H = B_H\setminus \{ \ab_1,\ab_2\}$.

Assume $\{\ab_1, \ab_2\} \subsetneq B_H$. Let $\bb \in G(\omega_H)$. The only lattice points on the boundary of the parallelogram $\overline{P_H}$ are $0, \ab_1, \ab_2, \ab_1+\ab_2$. None of them  is in $G(\omega_H)$, under our hypothesis. Thus $\bb \in \relint P_H$. If, on the contrary, $\bb \notin B_H$, then $\bb$ is the sum of at least two elements in $B_H$, out of which at least one is not in $\omega_H$, i.e. the latter is $\ab_1$ or $\ab_2$. This implies that $\bb \notin P_H$, a contradiction. Consequently, $G(\omega_H)\subseteq B_H \sect \omega_H$. The reverse inclusion is clear. 
\end{proof}

We refer to the  monographs \cite{BH}, \cite{BG}, \cite{Villarreal}, \cite{Ziegler}, \cite{Fulton}  for more background on affine semigroups, their semigroup rings, rational cones and the connections with algebraic geometry.

\section{A regular sequence in $K[H]$  and slim  semigroups}
\label{section:regular}
 
Throughout this section $H$ is a semigroup in $\Nc_d$ having the extremal rays $\ab_1,\dots, \ab_d$, and $R=K[H]$. 
The main result is  Theorem~\ref{tricky}  where we present    equivalent conditions for the existence of a convenient reduction  for the graded maximal ideal of $K[H]/(\xb^\bb)$, where $\bb$ is any nonzero element in $H$.
 This result motivates us to introduce the class of slim semigroups. 

The following lemma plays a crucial role in the proof of Proposition~\ref{regular} and in several other arguments in this paper.

\begin{Lemma}
	\label{b-in-omega}
	Let $\bb=\sum_{i=1}^d \lambda_i \ab_i$, with   $\lambda_1, \dots, \lambda_d \in \RR$.
	\begin{enumerate}
		\item [(a)] Let $n_i=\lfloor \lambda_i \rfloor +1$ for $i=1,\dots, d$. Then   $(\sum_{i=1}^d n_i  \ab_i) -\bb \in \omega_H$.
		\item [(b)]
		If $\bb \in \overline{P_H}$ then $ (\sum_{i=1}^d \ab_i) -\bb \in H$, and in particular, if $\bb \in P_H$,  then $(\sum_{i=1}^d \ab_i) -\bb \in \omega_H$.
	\end{enumerate}
\end{Lemma}

\begin{proof}
	For (a) we note that $(\sum_{i=1}^d n_i  \ab_i) -\bb= \sum_{i=1}^d (1-\{ \lambda_i \}) \ab_i$ and $0<1-\{\lambda_i \}\leq 1$ for all $i$, hence the sum of interest is in $\omega_H$. Here we denoted $\{ \lambda_i \}=\lambda_i- \lfloor \lambda_i \rfloor$ for all $i$. Part (b) follows immediately.
\end{proof}

\begin{Proposition}
	\label{regular}
For any   $\bb\neq 0$ in $H$, the sequence  $\xb^\bb,\xb^{\ab_1}-\xb^{\ab_2},  \ldots, \xb^{\ab_1}-\xb^{\ab_d}$ is a regular sequence on $R$.
\end{Proposition}

\begin{proof}
	In order to simplify notation we set $u=\xb^\bb$ and $v_i=\xb^{\ab_i}$ for $i=1,\ldots,d$. Let $I=(u,v_1-v_2,\ldots,v_1-v_d)$.
	We may write $\bb=\sum_{i=1}^d \lambda_i \ab_i$ with $\lambda_i \geq 0$ for $i=1, \dots, d$.  We denote $n_i=\lfloor \lambda_i \rfloor +1$ for all $i$ and we set $N=\sum_{i=1}^d n_i$.

	We will show that $v_i^N \in I$ for $i=1,\ldots,d$.
	Since $v_i-v_j\in I$ for all $i$ and $j$, it follows by symmetry that  it is enough to show that $v_1^N\in I$.
	
	We write
	\begin{eqnarray*}
		v_1^N &=& (v_1^{n_2}-v_2^{n_2}) \cdot v_1^{N-n_2}+ v_1^{N-n_2}v_2^{n_2} \\
		&=&  (v_1^{n_2}-v_2^{n_2}) \cdot v_1^{N-n_2}+ v_1^{N-n_2-n_3}v_2^{n_2}(v_1^{n_3}-v_3^{n_3})+ v_1^{N-n_2-n_3}v_2^{n_2}v_3^{n_3} \\
		&=& \sum_{i=2}^d v_1^{N-\sum_{j=2}^i n_j} v_2^{n_2}\cdots v_{i-1}^{n_{i-1}} (v_1^{n_i}-v_i^{n_i}) + v_1^{n_1}\cdots v_d^{n_d},
	\end{eqnarray*}
	hence by Lemma \ref{b-in-omega} it follows that $v_1^N \in I$.
	
	Since $H$ is a  normal simplicial semigroup, $v_1, \dots, v_d$ is a regular  sequence  in $R$, hence $v_1^N, \dots, v_d^N$ is a regular sequence in $R$, as well. Since $R$ is  a Cohen-Macaulay ring of dimension $d$ we get that  $v_1^N, \dots, v_d^N$ is also  a system of parameters for $R$. Thus
	$0 < \lambda (R/I) \leq \lambda(R/(v_1^{N}, \dots,v_d^N)) <\infty$, which implies that $u, v_1-v_2, \dots, v_1-v_d $ is a system of parameters, and consequently a regular sequence for $R$.
	Here $\lambda(M)$  denotes the length of an $R$-module $M$.
\end{proof}

In the sequel, our aim is to find a reduction ideal for the  graded maximal ideal $\mm$ of $R$,  modulo the ideal $\xb^\bb R$, for any $\bb\in H\setminus\{0\}$. In this order, we need the following lemma  which is interesting on its own.

\begin{Lemma}
	\label{lemma:k-for-u}
  For any $\bb$ in $H$, there exists  a positive integer $k$ such  that for all $\cb_1, \dots, \cb_k$ in $\omega_H$,  one has   $\cb_1+\dots+ \cb_k \in \bb+ H$.
\end{Lemma}

\begin{proof}
	Assume  $\nb_1, \dots, \nb_r \in \ZZ^d$ are   normal vectors to the support hyperplanes of the  facets of the cone $C$ such that
	$$
	C=\{\xb\in \RR^d: \langle \xb, \nb_i \rangle \geq 0, \text{ for all } i=1, \dots, r \}.
	$$
	
	The map $\sigma: \RR^d \to \RR^r$ given by
	$$
	\sigma(\hb)= (\langle \hb, \nb_1 \rangle, \dots, \langle \hb, \nb_r \rangle), \text{ for all } \hb \in \RR^d
	$$
	is clearly $\RR$-linear and $\sigma(H) \subseteq \NN^r$.
	
	Let $k_0= \max \{ \langle \bb, \nb_j \rangle: j=1, \dots, r \}$.  For any integer $k> k_0$,  any  $\cb_1, \dots, \cb_k \in H \sect\; \relint C$, and any $1\leq j \leq r$ ,  the $j$-th component of $\sigma(\cb_1+\dots +\cb_k- \bb)$ equals $(\sum_{i=1}^k \langle \cb_i, \nb_j \rangle ) - \langle \bb, \nb_j \rangle \geq k-k_0 >0$,   hence   $\cb_1+\dots+ \cb_k \in \bb+ H$.
\end{proof}

\begin{Theorem}
	\label{tricky}
Let $R=K[H]$, $J=(\xb^{\ab_i}-\xb^{\ab_j}\:\; i,j =1,\ldots,d)R$ and  $0\neq \bb \in \omega_H$. 
  Then the following statements are equivalent:
	\begin{enumerate}
		\item [(i)] there exists an integer $k$ such that $\mm^{k+1}=J\mm^k$ modulo the ideal $\xb^\bb R$;
		\item [(ii)] $\sum^d_{j=1}[\cb]_j\geq 1$, for any $\cb\in B_H \setminus\omega_H$.
 	\end{enumerate} 
\end{Theorem}

\begin{proof} 
 We denote $\{\cb_1,\dots,\cb_r\}=B_H\setminus(\{\ab_1,\dots,\ab_d\}\cup\omega_H)$. 

(i)$\implies$(ii):   
For any $\cb$ in $H$ we set $l(\cb)=\sum^d_{j=1}[\cb]_j$. Clearly,  $l(\ab_i)=1$ for $i=1,\dots,d$, so if $r=0$, we are done.
Assuming  $r>0$, we pick $t$ such that $l(\cb_t)=\min\{l(\cb_j): j=1,\dots,r\}$.
Let $k>0$ so that $\mm^{k+1}+\xb^\bb R= J \mm^k+\xb^\bb R$, i.e. $\mm^{k+1}\subseteq J \mm^k+\xb^\bb R$.

As $\cb_t \notin \omega_H$, there exists $1\leq i_0 \leq d$ with $[\cb_t]_{i_0}=0$, hence $(k+1) \cb_t \notin \bb+H$.
From $\xb^{(k+1)\cb_t} \in J\mm^k+\xb^\bb R$ we get that
$$
\xb^{(k+1)\cb_t}=\xb^{\ab_j}\xb^{\ub_1+\dots+\ub_k},
$$
for some $1\leq j \leq d$ and $\ub_1,\dots, \ub_k \in H\setminus (\omega_H \union \{0\})$.
 Consequently, 
\[
(k+1)l(\cb_t)=1+\sum^k_{i=1}l(\ub_i)\geq 1+k\min\{1,l(\cb_t)\},
\]
which implies $l(\cb_t)\geq 1$. 

(ii)$\implies$(i): 
Let $u=\xb^\bb$	and  $v_i=\xb^{\ab_i}$ for $i=1,\dots, d$. We decompose $\bb=\sum_{i=1}^d \lambda_i \ab_i$ with $\lambda_i \geq  0$ and we set $n_i=\lfloor \lambda_i \rfloor +1$ for all $i=1,\dots, d$ and $N=\sum_{i=1}^d n_i$.
	
	We claim that for any positive integer $t$, any $i_1, \dots, i_t \in \{ 1,\dots, d \}$ and any $v\in\{v_1, \dots, v_d\}$ one has
	\begin{equation}\label{eq:prod}
	v_{i_1}\cdots v_{i_t} \in J \mm^{t-1} +v^tR.
	\end{equation}
	Indeed, this is a consequence of the following equations:
	\begin{eqnarray*}
		v_{i_1}\cdots v_{i_t} &=& (v_{i_1}-v)\cdot v_{i_2}\cdots v_{i_t} +v\cdot v_{i_2}\cdots v_{i_t} \\
		& = &   (v_{i_1}-v)\cdot v_{i_2}\cdots v_{i_t} +  v(v_{i_2}-v)\cdot v_{i_3}\cdots v_{i_t}+ v^2 v_{i_3}\cdots v_{i_t} \\
		&= & \Sigma_{j=1}^d  v^{j-1} \cdot (v_{i_j}-v)\cdot v_{i_{j+1}}\cdots v_{i_t} + v^t.
	\end{eqnarray*}
	
	Now let $i_1, \dots, i_N \in \{1,\dots, d \}$.  In the product $v_{i_1}\dots v_{i_N}$ we apply \eqref{eq:prod}   to the first $n_1$ terms, then to the next $n_2$ terms, etc. and  we obtain that
	\begin{equation}
	\label{prod2}
	v_{i_1}\dots v_{i_N} \in \prod_{i=1}^d (J\mm^{n_i-1}, v_i^{n_i}) \subseteq  (J\mm^{N-1}, \prod_{i=1}^d v_i^{n_i}) \subseteq (J\mm^{N-1}, uR),
	\end{equation}
	where for the last inclusion we used  Lemma \ref{b-in-omega}.

For $1\leq i\leq r$ we may write $\cb_i=\sum_{j=1}^d \frac{p_{ij}}{q_i}\ab_i$, where $q_i, p_{ij}$ are nonnegative integers and $q_i >0$. Hence, $Nq_i \cdot \cb_i=\sum_{j=1}^d Np_{ij} \ab_j$, where by the hypothesis of (ii) we have $N\leq N q_i \leq \sum_{j=1}^d Np_{ij}$. Thus, using \eqref{prod2} we derive
\begin{equation*}
\label{someeq}
(\xb^{\cb_i})^{Nq_i}= \prod_{j=1}^d (\xb^{\ab_j})^{Np_{ij}} \in (J\mm^{\sum_{j=1}^{d} N p_{ij}-1}, uR) \subseteq (J\mm^{N q_i -1}, uR).
\end{equation*}
We set  $N_1=\max\{ N q_i:1\leq i\leq r \}$ if $r>0$, otherwise we let $N_1=N$. Then 
\begin{equation}
\label{somesthingelse}
(\xb^\cb)^{N_1} \subseteq (J\mm^{N_1-1}, uR), \text{ for all } \cb \in B_H\setminus \omega_H.
\end{equation}

	Let $k_0$ be a positive integer satisfying the conclusion of Lemma \ref{lemma:k-for-u} for the  element $\bb$. We set $k=k_0+N_1 \cdot |B_H\setminus \omega_H|-2$.
	
	Let $w$ be any product  of $k+1$ monomial generators of $\mm$. If the exponents of at least $k_0$ of them  are in $\relint C$, then by the choice of  $k_0$ we get that $w \in uR$. Otherwise,  we may write  $w= (\xb^{\cb})^{N_1}  \cdot w'$ for some $\cb\in B_H\setminus\omega_H$ and $w' \in \mm^{k+1-N_1}$. In the latter case, using \eqref{somesthingelse} we infer that $w \in J \mm^k + uR$.
	This shows that $\mm^{k+1} +uR= J\mm^k + uR$, which completes the proof.
\end{proof}

The semigroups $H$ satisfying the equivalent conditions of Theorem~\ref{tricky}  deserve a special name.	
	
	\begin{Definition}
		{\em
		A   semigroup $H\in \Nc_d$ is called {\em slim} if $\sum^d_{i=1}[c]_i \geq 1$, for any $\cb\in B_H \setminus \omega_H$. 
	}
	\end{Definition}
We will denote  by $\Hc_d$ the class of slim semigroups in $\NN^d$.

\begin{Remark}{\em
\label{rem:many-slim}
Let $H \in \Nc_d$. If $B_H\setminus \omega_H=\{\ab_1,\dots, \ab_d\}$, then $H$ is slim.

In particular, by Lemma~\ref{lemma:omega}, any normal semigroup in $\NN^2$ is slim, so $\Nc_2=\Hc_2$.}
\end{Remark}

\begin{Proposition}
\label{slim-3d}
Let $H\in \Nc_3$. Then $H$ is slim if and only if $[\cb]_1+[\cb]_2+[\cb]_3=1$ for any $\cb\in B_H\setminus \omega_H$.
\end{Proposition}
\begin{proof} Assume $H$ is slim.
If $B_H\setminus \omega_H =\{\ab_1,\ab_2, \ab_3\}$, there is nothing left to prove.
Assume there exists $\cb\in B_H\setminus (\omega_H \cup \{\ab_1, \ab_2, \ab_3\})$ such that $\sum_{i=1}^3 [\cb]_i >1$.
Without loss of generality, we may assume that $0<[\cb]_i <1$ for $i=1,2$ and  $[\cb]_3=0$.
Arguing as in the proof of Lemma~\ref{b-in-omega} we obtain that $0\neq \ab_1+\ab_2-\cb \in \NN \ab_1+\NN \ab_2 \subset H$.
We may write $\ab_1+\ab_2-\cb=\bb+\hb$, where $\bb\in B_H\setminus \omega_H$ and $\hb \in H$.
Then $\sum_{i=1}^3 [\bb]_i = [\bb]_1+[\bb]_2 \leq \sum_{i=1}^2 [\ab_1+\ab_2-\cb]_i= 2-\sum_{i=1}^2 [\cb_i] <1$, which is false since $H$ is slim.
\end{proof}

\begin{Example}
\label{ex:nonslim}{\em
The semigroup $L\in \Nc_3$ with the extremal rays $\ab_1=(11,13,0)$, $\ab_2=(3,4,0)$ and $\ab_3=(0,0,1)$ is not slim. Indeed, $B_L=\{\ab_1, \ab_2, \ab_3, (4,5,0), (5,6,0)\}$ (compare with Example~\ref{bott-Ul}) and it is easy to check that $\sum_{i=1}^3 [(4,5,0)]_i=4/5<1$.}
\end{Example}

\section{Ulrich elements and the almost Gorenstein property}

The theory of almost Gorenstein rings has its origin in the theory of the almost symmetric numerical semigroups in \cite{BF}. If $R$ is the semigroup ring of a numerical semigroup,  then the semigroup is almost symmetric, if and only if  there  exists an exact sequence
\begin{eqnarray}
\label{basic}
0\to R \to  (\omega_R)(-a) \to  E \to 0,
\end{eqnarray}
where $E$ is annihilated by the graded maximal ideal of $R$, see \cite{Herzog-Watanabe}. Here $\omega_R$ denotes the canonical module of $R$ and $-a$ the smallest degree of a generator of $\omega_R$, i.e. $-a =\min\{k : (\omega_R)_k\neq 0\}$

In \cite{GMT} the $1$-dimensional positively graded rings which admit such an exact  sequence are called almost Gorenstein.

Goto et al. \cite[Definition 8.1]{GTT-2015} extended the concept of the  almost Gorenstein property to rings of higher dimension:
let $R$ be a positively graded  Cohen--Macaulay  $K$-algebra with $a$-invariant $a$.  Then  $R$ is called {\em graded almost Gorenstein}, if there exists an exact sequence like in (\ref{basic}),
where $E$ is an Ulrich module.

Ulrich modules are defined as follows: let $(R, \mm, K)$  be a local (or positively graded)  ring with (graded) maximal ideal $\mm$, and let $M$ be a (graded) Cohen-Macaulay module  over $R$. Then the minimal number of generators $\mu(M)$ of $M$  is bounded by the multiplicity $e(M)$ of $M$. The module $M$ is called an {\em Ulrich module},  if  $\mu(M) = e(M)$.  In  \cite{Ul} Ulrich asked whether any Cohen--Macaulay ring admits an Ulrich module $M$  with $\dim M=\dim R$. At present this question is still open, and  has an affirmative answer for example when $R$ is a hypersurface ring  \cite{HB}.

In the case of almost symmetric numerical semigroup rings, the module $E$ in  the exact  sequence (\ref{basic}) is of Krull dimension zero. A graded module $M$ with $\dim M=0$ is  Ulrich if and only if $\mm M=0$.  Thus the above definition \cite[Definition 8.1]{GTT-2015} is a natural extension of $1$-dimensional almost Gorenstein rings to higher dimensions.

\medskip

We propose the  following  multigraded version of the almost Gorenstein property for normal semigroup rings.

\begin{Definition}{\em
	\label{def:ag}
Let $H$ be a normal affine semigroup and $R=K[H]$. For $\bb \in\omega_H$ consider the following exact sequence
\begin{eqnarray}
\label{basicexact}
0\to R \to  \omega_R(\bb) \to  E \to 0,
\end{eqnarray}
where $1\in R$ is mapped to $u=\xb^\bb$ and  $E =\omega_R/uR$.
Then  $\bb$ is called  an \textit{Ulrich element} in $H$, if $E$ is an Ulrich $R$-module.

If $H$ admits  an Ulrich element $\bb$, then we call the ring $R$ {\em almost Gorenstein with respect to $\bb$}, or  simply  AG if $H$ has an Ulrich element.
}
\end{Definition}

\begin{Theorem}
	\label{criterion} 
Let   $H\in\Hc_d$  with extremal rays $\ab_1,\ldots,\ab_d$,  and let $\mm$ be the graded maximal ideal of $R=K[H]$.
	Let   $\bb \in \omega_H$,  $u=\xb^\bb$ and $J=(\xb^{\ab_i}-\xb^{\ab_j}\:\; i,j =1,\ldots,d)R$. 
    Then $\bb$  is an Ulrich element  in $H$ if and only if 
\begin{eqnarray}
\label{Kinfinite}
\mm \omega_R\subseteq (uR,J\omega_R).
	\end{eqnarray}

\end{Theorem}

\begin{proof}
	Since $uR$ and $\omega_R$ are Cohen--Macaulay  $R$-modules  of dimension $d$, we see (keeping the notation from \eqref{basicexact}) that  $\depth E\geq d-1$, and since $uR$ and $\omega_R$  are rank $1$ modules, we deduce that $\Ann(E)\neq 0$. Therefore, $\dim E\leq d-1$, and this implies that  $E$ is a Cohen--Macaulay $R$-module of dimension $d-1$.

	Suppose that (\ref{Kinfinite}) holds. By \cite[Proposition 2.2.(2)]{GTT-2015},  it follows that $E$  is an Ulrich module  since  (\ref{Kinfinite}) implies that  $\mm E=JE$ and since $J$ is generated by $d-1(=\dim E)$ elements, namely by the elements $f_j=\xb^{\ab_1}-\xb^{\ab_j}$ with  $j=2,\ldots,d$. Thus $\bb$ is an Ulrich element in $H$.

	Conversely, assume that $\bb$ is an Ulrich element. Then $E$ is an Ulrich module, and therefore $\length (E/\mm E)=e(E)$. It follows from Theorem~\ref{tricky} that $J$ is a reduction ideal  of $\mm$ with respect to $E$. Thus by \cite[Lemma 4.6.5]{BH}, $e(E)=e(J,E)$, where $e(J,E)$ denotes the Hilbert-Samuel multiplicity of $E$ with respect to $J$.  Since $E$ is Cohen--Macaulay of dimension $d-1$, and since $J$ is generated by the $d-1$ elements $f_2,\ldots,f_d$  and $\length (E/JE)<\infty$, we see that $f_2,\ldots,f_d$ is a regular sequence on $E$. Thus \cite[Theorem 4.7.6]{BH} implies that $e(J,E)=\length(E/JE)$. Hence, $\length (E/\mm E)=\length (E/JE)$. Since $JE\subseteq \mm E$, it follows that $\mm E=JE$,  and this implies (\ref{Kinfinite}).
\end{proof}
 
\begin{Remark}
{\em
It follows from the proof of Theorem~\ref{criterion} that if $H\in \Nc_d$ and \eqref{Kinfinite} holds for some ideal 
$J\subset \mm$, generated by $d-1$ elements, then $\bb$ is an Ulrich element in $H$.
}
\end{Remark}

\begin{Example} 
\label{gore-regular}{\em  (Ulrich elements in Gorenstein and regular rings)
\begin{enumerate}
\item[(a)] If $K[H]$ is a Gorenstein ring and $G(\omega_H)=\{ \bb\}$, then $\omega_R=\xb^\bb R$, hence \eqref{Kinfinite} holds and $\bb$ is an Ulrich element in $H$.
\item[(b)] Assume $K[H]$ is a regular ring with $\ab_1,\dots, \ab_d$ the extremal rays of $H$. Set $\cb=\sum_{i=1}^d \ab_i$.
 Then $\ab_i+\cb$ is an Ulrich element in $H$ for any $i=1,\dots, d$. Indeed, since $\mm=(\xb^{\ab_j}:1\leq j \leq d)$  and $\xb^{\ab_j+\cb} =\xb^{\cb}(\xb^{\ab_j} - \xb^{\ab_i} ) +\xb^{\cb+\ab_i}$ for  $j=1,\dots, d$, we have that \eqref{Kinfinite} is verified for $\bb=\cb+ \ab_i$.
\end{enumerate}
}
\end{Example}

\begin{Example}{\em
	\label{ex:nice}
	Let $H\in \Hc_2$ having the extremal rays $\ab_1=(11,2)$ and $\ab_2= (31,6)$. A computation with Normaliz \cite{Normaliz} shows that the Hilbert basis of $H$ is
	$$
	B_H=\{ \ab_1, \ab_2, \bb=(16,3), \cb_1=(21,4), \cb_2=(26,5)\}.
	$$
	Moreover, $\bb$, $\cb_1$, $\cb_2$ are the only nonzero lattice points in $P_H$, and they all lie on the line $y=(x-1)/5$ passing through $\ab_1$ and $\ab_2$.
	Comparing componentwise, we have
	$$
	\ab_1 \preceq \bb \preceq \cb_1 \preceq \cb_2 \preceq \ab_2.
	$$
          We note that  $\ab_1 +\ab_2= \bb+\cb_2=2 \cb_1$.
	It is also straightforward to check that in $K[H]$ we have
	\begin{eqnarray*}
		\xb^{\ab_1}  \xb^{\cb_1}& = &(\xb^\bb)^2, \;
		\xb^{\ab_1}  \xb^{\cb_2} = \xb^\bb   \xb^{\cb_1}, \;
		\xb^{\cb_1}  \xb^{\cb_1} = \xb^\bb   \xb^{\cb_2}, \;
		\xb^{\cb_1}  \xb^{\cb_2} = \xb^\bb   \xb^{\ab_2}, \\
		\xb^{\cb_2} \xb^{\cb_2} &=&  \xb^{\ab_2} \xb^{\cb_1}= (\xb^{\ab_2}- \xb^{\ab_1}) \xb^{\cb_1}+ \xb^{\ab_1} \xb^{\cb_1}= (\xb^{\ab_2}- \xb^{\ab_1}) \xb^{\cb_1}+    (\xb^\bb)^2, \\
	   \xb^{\ab_2} \xb^{\cb_2} &=& (\xb^{\ab_2}-\xb^{\ab_1}) \xb^{\cb_2} + \xb^{\ab_1} \xb^{\cb_2} =
(\xb^{\ab_2}-\xb^{\ab_1}) \xb^{\cb_2}+\xb^{\bb} \xb^{\cb_1}.  
	\end{eqnarray*}
	Using    Theorem \ref{criterion} we conclude  that  $\bb$ is an Ulrich element in $H$, and hence $K[H]$ is AG.
}
\end{Example}

In the following special case, the possible Ulrich elements can be identified.

\begin{Proposition}\label{1}
Let $H$ be a semigroup in $\Hc_d$ whose nonzero elements have all the entries positive, and assume that $(1, \dots, 1)\in\omega_H$. If $H$ has an  Ulrich element  $\bb$, then $\bb=(1, \dots, 1)$.
\end{Proposition}

\begin{proof}
We set $\bb'=(1,\dots, 1)$.
Assume, on the contrary,  that $\bb\neq \bb'$. Then  by the criterion in Theorem \ref{criterion} and using the same notation, we get that   $\xb^{\bb'} \cdot \xb^{\bb'} \in (\xb^\bb R, J \omega_R)$. This implies that  $(2,\dots, 2)=2\bb'= \bb +\cb$ for some $\cb \in H$, or that $2\bb'= \ab_i+\hb$ for some $1\leq i \leq d$ and $\hb\in \omega_H$, $\hb\neq \bb'$.
	Since $(1,\dots, 1)$ is the smallest element of $\omega_H$ when comparing vectors componentwise,
  at least one component of $\bb$ (respectively, of $\hb$)  is at least two, hence
	at least one component of $\cb$ (respectively, of $\ab_i$) is less than or equal to zero, which is false by the assumption that  all the entries of  nonzero elements of $H$ are positive.
\end{proof}


\section{The AG property for normal semigroups in dimension $2$}
\label{section:dim2}

As  mentioned before, all $2$-dimensional normal affine  semigroups are slim.
For them, in Theorem~\ref{two} we make more concrete the criterion for Ulrich elements given in Theorem~\ref{criterion}. 
We first prove  a couple of lemmas.

Throughout this section, unless otherwise stated, $H$ is a semigroup in $\Nc_2=\Hc_2$ with extremal rays $\ab_1$ and $\ab_2$.
We denote by $\mm$ the graded maximal ideal of $R=K[H]$.

\begin{Lemma}
	\label{nice}
	Let $\bb$ be an element in $B_H\setminus\{\ab_1,\ab_2\}$. 
	For any $\cb \in \omega_H $ such that $\cb \notin \bb+H$,  there exists $t\in \{1,2\}$ such that $\cb+\ab_t \in \bb+H$.
\end{Lemma}
\begin{proof}
	Let $C$ be the cone with  the extremal rays $\ab_1$ and $\ab_2$.
	Let $\nb_1$ and $\nb_2$ be vectors normal to the facets of the cone $C$   such that $\langle \ab_i, \nb_i\rangle=0$ for $i=1,2$ and 
	$\xb \in C$ if and only if $\langle \xb, \nb_1 \rangle \geq 0$ and $\langle \xb, \nb_2 \rangle \geq 0$.
	
Since $\cb \notin \bb+H$ it follows that $\cb-\bb \notin C$.  We may assume that $\langle \cb-\bb, \nb_1\rangle<0$, and   claim  then that  $\cb+\ab_2 \in \bb+H$. Indeed,
\[
\langle \cb+\ab_2-\bb,\nb_1\rangle = \langle \cb,  \nb_1\rangle +\langle \ab_1+\ab_2-\bb,\nb_1\rangle>0,
\]
since $\langle \ab_1+\ab_2-\bb,\nb_1\rangle>0$, by Lemma~\ref{b-in-omega}, and
\[
\langle \cb+\ab_2-\bb,\nb_2\rangle = \langle \cb-\bb,  \nb_2\rangle >0,
\]
since otherwise $\cb-\bb\in -C=\{-\ab\:\; \ab\in C\}$, a contradiction  to $\bb \in B_H$.
\end{proof}

\begin{Lemma}
	\label{c-inside-I}
	Let $\bb$ belong to $B_H\setminus\{\ab_1,\ab_2\}$.
 We set  $I=(\xb^\bb R,(\xb^{\ab_1}-\xb^{\ab_2})\omega_R)\subset R$. 
Let $\cb \in \omega_H$. The following conditions are equivalent:
	\begin{enumerate}
		\item[(a)] $\xb^\cb\in I$;
		\item[(b)]   $\cb \in (\bb+H) \union (\ab_1+\omega_H) \union (\ab_2+\omega_H)$;
		\item[(c)]  $\cb \in (\bb+H) \union (\ab_1+ H) \union (\ab_2+ H)$.
	\end{enumerate}
\end{Lemma}

\begin{proof}
	(a) \implies (b): Note that   $\xb^\cb \in \xb^\bb R$ if and only if  $\cb \in \bb+H$. If $\xb^\cb\notin \xb^{\bb}R$, then there exist $0\neq f$ in $\omega_R$ and $g$ in $R$ such that
	$$
	\xb^\cb = (\xb^{\ab_1} - \xb^{\ab_2})\cdot f + \xb^\bb \cdot g.
	$$
	Therefore, there exists a monomial $\xb^\ab$ in $\omega_R$ such that $\xb^\cb=\xb^{\ab_1} \cdot \xb^{\ab}$ or $\xb^\cb=\xb^{\ab_2} \cdot \xb^{\ab}$, equivalently  $\cb \in (\ab_1 + \omega_H)\union (\ab_2 + \omega_H)$.
	
	(b) \implies (a):  If $\cb \in \bb +H$ then clearly $\xb^\cb \in I$. Assume $\cb \notin \bb+H$. By symmetry, it is enough to consider the case  $\cb \in \ab_1 +\omega_H$.
	
	By Lemma \ref{nice}, since $0\neq \cb-\ab_1 \in \omega_H$, $\cb-\ab_1  \notin \bb +H$  and  $(\cb-\ab_1)+\ab_1=\cb \notin \bb+H$ it follows that $\cb-\ab_1 +\ab_2 \in \bb +H$.
	
	As we may write
	$$
	\xb^\cb= \xb^{\cb-\ab_1} \cdot (\xb^{\ab_1}-\xb^{\ab_2})+\xb^{\cb-\ab_1+\ab_2},
	$$
	we conclude that $\xb^\cb \in I$.
	
	(b) \implies (c) is trivial.
	
	For (c) \implies (b) it is enough to consider the case when $\cb \notin \bb+H$. We may assume
	 $\cb \in \ab_1 + H$. If $\cb \notin \ab_1 + \omega_H$, then there exists a positive integer $n$ such that
	either $\cb-\ab_1=n \ab_1$, or $\cb-\ab_1=n \ab_2$. In the first  case we get that $\cb =(n+1)\ab_1 \notin \omega_H$, a contradiction. In the second  case we get that $\cb = \ab_1+ n \ab_2 \in \bb+H$, since $\ab_1+\ab_2-\bb \in H$ by Lemma \ref{b-in-omega}. This  is again a contradiction.  Thus $\cb \in \ab_1+ \omega_H$.
	\end{proof}

\begin{Theorem}
	\label{two}
An element  $\bb\in B_H\setminus\{\ab_1,\ab_2\}$ is  an Ulrich element in $H$, if and only if
	$$
	\cb_1+\cb_2 \in    (\bb+H) \union (\ab_1+ H) \union (\ab_2+ H),  \text{ for all } \cb_1,\cb_2\in B_H.
	$$
\end{Theorem}

\begin{proof}
Let $B_H=\{\ab_1,\ab_2,\cb_1,\ldots,\cb_m\}$. Then $\mm=(\xb^{\ab_1},\xb^{\ab_2}, \xb^{\cb_1},\ldots, \xb^{\cb_m})$ and $\omega_R= (\xb^{\cb_1},\ldots, \xb^{\cb_m})$.

If $\bb$ is an Ulrich element, then $\mm \omega_R\subseteq (\xb^{\bb} R,(\xb^{\ab_1}-\xb^{\ab_2})\omega_R)$,  and therefore $\xb^{\cb_i}\xb^{\cb_j}\in (\xb^{\bb} R,(\xb^{\ab_1}-\xb^{\ab_2})\omega_R)$ for all $i,j$. Thus the desired conclusion follows from Lemma~\ref{c-inside-I}.

Conversely, let $\xb^{\cb}\in \mm\omega_R$. Then $\cb=\cb_i+\cb_j+h$, or $\cb=\ab_i+\cb_j+h$ for some  $h\in H$. In both cases our assumptions imply that $\cb \in (\bb+H) \union (\ab_1+ H) \union (\ab_2+ H)$. Thus $\xb^{\cb} \in (\xb^{\bb} R,(\xb^{\ab_1}-\xb^{\ab_2})\omega_R)$, by Lemma~\ref{c-inside-I}. This shows that $\bb$ is an Ulrich element in $H$.
\end{proof}


\begin{Example}
{\em
	\label{ex:notnice}
	Let $H$ be the semigroup in $\Hc_2$ with the extremal rays $\ab_1=(5,2)$ and $\ab_2= (2,5)$. Then the Hilbert basis of $H$ is
	\[
	B_H=\{ \ab_1, \ab_2, \cb_1=(1,1), \cb_2=(2,1), \cb_3=(1,2)\}.
	\]
	Using  Theorem~\ref{two}, we may check that none of $\cb_1,\cb_2$ or $\cb_3$ is an Ulrich element in $H$. The same conclusion could be reached by using  Proposition~\ref{1} together with Theorem~\ref{ulrich11}.
}
\end{Example}

Here is one immediate application of Theorem~\ref{two}.

\begin{Proposition}\label{allulrich}
Let $H\in \Hc_2$ such that  $\cb+\cb'\notin P_H$ for all $\cb, \cb' \in B_H\cap P_H$.
Then any $\bb \in B_H \cap P_H$ is an Ulrich element in  $H$.
\end{Proposition}

\begin{proof}
By the hypothesis, if $\cb, \cb' \in B_H\cap P_H$ then $\cb+\cb' \in (\ab_1+H)\cup (\ab_2 +H)$. Theorem \ref{two} yields the conclusion.
\end{proof}

	One may check that the semigroup $H$ in Example \ref{ex:nice} satisfies the hypothesis of Proposition~\ref{allulrich}, hence $H$ admits three Ulrich elements.

In the following example there is exactly  one Ulrich element in the Hilbert basis of $H$.

\begin{Example}
{\em 
       \label{bott-Ul}
	For the semigroup $H\in \Hc_2$ with $\ab_1=(11,13)$ and $\ab_2=(3,4)$,  a Normaliz  (\cite{Normaliz}) computation 
shows that  $B_H=\{\ab_1,\ab_2,\cb_1=(4,5),\cb_2=(5,6)\}$. 
	We note that the points $2\cb_2-\cb_1=(6,7)$ and $2\cb_2-\ab_2= (7,8)$ are not in $\omega_H$ since the slope of the 
line through the origin and each of these respective points  is not in the interval $(13/11,4/3)$. Also, clearly, $2\cb_2-\ab_1 =(-1,-1) \notin H$.
Therefore, by Theorem~\ref{two} we get that $\cb_1$ is not an Ulrich element in $H$.

On the other hand, since $2\cb_1=(8,10)=\cb_2+\ab_2$ and $B_H\setminus \{\ab_1, \ab_2\}=\{\cb_1,\cb_2\}$,  by Theorem~\ref{two} we conclude that  $\cb_2$ is an Ulrich element in $H$.
}
\end{Example}

\section{Bottom elements as Ulrich elements in dimension $2$}
\label{sec:bottom}

In the multigraded situation   which we consider  in Definition~\ref{def:ag},  there is in general no distinguished multidegree with $(\omega_{K[H]})_\bb\neq 0$.
 Inspired by Proposition~\ref{1}, we are prompted to test the Ulrich  property for elements in $\omega_H$ with smallest entries. First we present the following lemma.

\begin{Lemma}
	\label{exists}
	 For any $H\in \Hc_2$, the set $\omega_H$  has a unique minimal element with respect to the componentwise partial ordering.
\end{Lemma}

\begin{proof}
	Let $C$ be the cone of $H$, and let $\ab=(a_1,a_2)$ and $\bb=(b_1,b_2)$ be points  in the relative interior of $C$. We claim that $\ab\wedge \bb=(\min\{a_1,b_1\}, \min\{a_2,b_2\})\in\relint C$. This will imply the existence of  the unique minimal element of $\omega_H$.
	
	For the proof of the claim, it is enough to consider the case when $a_1<b_1$ and $a_2>b_2$.
	Since $a_2/a_1 >b_2/a_1 > b_2/b_1$, it follows that the point in the plane with coordinates  $\ab\wedge \bb=  (a_1, b_2)$ lies inside the cone with vertex the origin and passing through the points with coordinates $\ab$ and $\bb$. Since the latter cone is in $\relint C$ it follows that  $\ab\wedge \bb \in \relint C$.
\end{proof}

 We call the unique minimal element of $\omega_H$ with respect to the componentwise partial ordering, {\em the bottom element} of $H$. 
\begin{Remark}
{\em
 For $H\in \Hc_2$, since the elements in $\omega_H$ have only nonnegative entries, it follows that the bottom element of $H$ is also the smallest element in $G(\omega_H)$ with respect to the componentwise order.  Moreover, if $K[H]$ is not a regular ring then the bottom element of $H$ is componentwise the smallest element in $B_H\setminus \{\ab_1,\ab_2\}$, see Lemma~\ref{lemma:omega}.
}
\end{Remark}

In arbitrary embedding dimension  we give the following definition.

\begin{Definition}
For $H\in \Hc_d$, an  element  $\bb\in \omega_H$ is called the {\em bottom element} of $H$ if $\cb-\bb \in \NN^d$ for all $\cb \in \omega_H$.
\end{Definition}

\begin{Remark}\label{rem:nobottom}
{\em	In general, a semigroup $H\in \Hc_d$ with  $d>2$ may  not have  a unique minimal element in $\omega_H$ with respect to the componentwise partial ordering  $\preceq$.
For instance, let $d=3$, $\ab_1=(5,3,1), \ab_2=(1,5,2),\ab_3=(8,3,5)$. Then, a calculation with Normaliz (\cite{Normaliz}) shows that
	\begin{eqnarray*}
		B_H&=&\{\ab_1,\ab_2,\ab_3, (1,2,1), (2,1,1), (2,2,1),  (2,5,2), (3,2,1), (3,2,2),\\
		&& (3,5,2), (3,5,3), (4,5,2), (5,2,3), (5,5,2), (5,5,4), (7,5,5)\}.
	\end{eqnarray*}
One can check that the vectors  $\nb_1=(19,11,-37), \nb_2=(-12, 17,9), \nb_3=(1,-9,22)$ are normal to the planes generated by $\ab_2$ and $\ab_3$, by $\ab_1$ and $\ab_3$, by $\ab_1$ and $\ab_2$  respectively. Also, that no element in $B_H\setminus \{\ab_1,\ab_2,\ab_3\}$ lies on any of the three planes just mentioned. Consequently, there are no inner lattice points on the faces of $\overline{P_H}$ and $G(\omega_H)= B_H\setminus \{\ab_1,\ab_2,\ab_3\}$.
	It follows that $\bb_1=(1,2,1)$ and $\bb_2=(2,1,1)$ are both minimal elements in $\omega_H$ 
 with respect to $\preceq$.
}
\end{Remark}

Using Theorem~\ref{two} we show that sometimes the bottom element may be the only Ulrich element in $B_H$.
 
\begin{Proposition}\label{twoinp}
Let $\bb$ be the bottom element of $H\in \Hc_2$. If $2\bb \in P_H$, then $\bb$ is the only possible Ulrich element in $B_H$.
\end{Proposition} 

\begin{proof}
Assume $\bb' \in B_H $ is an Ulrich element in $H$.
Then  $2\bb \in (\ab_1+H)\cup (\ab_2+H)\cup (\bb'+H)$, by Theorem \ref{two}.
	Since $2\bb \in P_H$, we get that $2\bb\in \bb'+H$, hence  $2\bb=\bb'+\hb$ for some $\hb \in P_H$.
	Moreover, comparing componentwise, $\bb\preceq \bb'$ and $\bb \preceq \hb$ since $\bb$ is the bottom element for $H$, thus $\bb'=\bb$.
\end{proof}

\begin{Remark}
{\em	In general, as Example~\ref{bott-Ul} shows,  even when the Hilbert basis of $H$ contains a unique Ulrich element, the latter need not  be  the bottom element. 
}
\end{Remark}

\medskip
 
In the following, we discuss when the bottom element $\bb$ of $H\in\Hc_2$ is  Ulrich.

\begin{Notation} 
\label{not}
{\em To avoid repetitions, in the rest of the section $H \in \Hc_2$ has the extremal rays
$\ab_1=(x_1,y_1)$ and $\ab_2=(x_2,y_2)$ such that ($y_1/x_1<y_2/x_2$ when $x_1, x_2 >0$) or  $x_2=0$. 

We define $H_1$ and $H_2$ to be the semigroups in $\Hc_2$ with the extremal rays $\ab_1$ and $\bb$, respectively $\ab_2$ and $\bb$.
 We   denote $\ZZ^2\cap\relint P_{H_i}$ by $H_i^*$, for $i=1,2$.}
\end{Notation}

\medskip
By an easy argument, the following proposition  presents a  class of semigroups in $\Hc_2$ with Ulrich  bottom element.
\begin{Proposition}\label{bb-ulrich}
	Let $\bb$ be the bottom element of $H$.  If $(x_2,y_1)\preceq\bb$, then  $\bb$ is an Ulrich element in $H$.
\end{Proposition}


\begin{proof}
If $K[H]$ is a regular ring, then $\bb$ is an Ulrich element in $H$, since $G(\omega_H)=\{\bb\}$.

Assume $K[H]$ is not a regular ring. Then $\bb\in G(\omega_H)= B_H\setminus \{\ab_1, \ab_2\}$, by Lemma~\ref{lemma:omega}(d).
Let $\cb_1, \cb_2 \in B_H\setminus \{\ab_1,\ab_2\}$, and $\cb_1+\cb_2=(c,d), \bb=(u,v)$.
 
If $(c,d)\in H_1$, then $(c,d)=r_1(x_1,y_1)+r_2(u,v)$ for some $r_1,r_2\in\RR_{\geq0}$. Since $d\geq 2v\geq y_1+v$, we have $r_1\geq 1$ or $r_2\geq 1$. Consequently, 		
	$(c,d)\in (\bb+H_1)\cup(\ab_1+H_1)\subset(\bb+H)\cup(\ab_1+H)$.
	
	A similar argument shows that if $(c,d)\in H_2$, then $(c,d)\in(\bb+H)\cup(\ab_2+H)$.
	The conclusion follows by 	Theorem~\ref{two}.
\end{proof}

\begin{Example}	\label{allalmost}{\em
      Let $H$ be the semigroup   with extremal rays $\ab_1=(a,1)$ and $\ab_2= (1,b)$, where $a,b\geq 2$.
Since $1/a<1<b$ we get that $\bb=(1,1)$ is in $\omega_H$ and it  is the bottom element in $H$.  Then Proposition \ref{bb-ulrich} implies that $\bb$ is   an Ulrich element in $H$.
}
\end{Example}

Clearly, $H=H_1\union H_2$ and $H_1\sect H_2=\NN \bb$. 
The following lemma states some nice properties regarding $H_1$ and $H_2$.

\begin{Lemma}\label{H12b}
 Let $\bb$ be the bottom element of $H$. Then
 \begin{enumerate}
 	\item[(a)] $\pb+\qb\in\bb+H$ for all $\pb\in H_1\setminus\{0\}$ and $\qb\in H_2\setminus \{0\}$.
 	\item[(b)] $(\bb+H)\union(\ab_1+H)\union(\ab_2+H)=H\setminus(H_1^*\cup H_2^*)$.
 \end{enumerate}
\end{Lemma}
\begin{proof}
(a). If $\pb-\bb\in H$ or $\qb-\bb\in H$, then clearly $\pb+\qb \in \bb +H$. Let us assume that $\pb-\bb\notin H$ and $\qb-\bb\notin H$.
Let $C'$ be the cone generated by the extremal rays $\pb,\qb$. Since
$\bb\in C'$,  $\bb=r_1\pb+r_2\qb$ for some $r_1,r_2\in\RR_{>0}$.
 If $r_1>1$, then $\bb-\pb=(r_1-1)\pb+r_2\qb$, hence $\bb-\pb\in C'\sect \omega_H$, a contradiction with $\bb$ the bottom element in $H$.
Therefore, $r_1\leq 1$, and also $r_2\leq 1$ by a similar argument.
 Now, $\pb+\qb-\bb=(1-r_1)\pb+(1-r_2)\qb \in C'\sect \ZZ^2 \subset H$.

(b).
Note that for any $\pb\in H$ we have
\begin{eqnarray*}
	\pb\in H_1\setminus H_1^* \Leftrightarrow \pb\in (\bb+H_1)\union (\ab_1+H_1),\\
	\pb\in H_2\setminus H_2^* \Leftrightarrow \pb\in (\bb+H_2)\union (\ab_2+H_2).
\end{eqnarray*}
Therefore, $H\setminus(H_1^*\cup H_2^*)  = (H_1\union H_2)\setminus(H_1^*\cup H_2^*) \subseteq (\bb+H)\union(\ab_1+H)\union(\ab_2+H)$. 

In order to check the reverse inclusion, let $\pb\in (\bb+H)\union(\ab_1+H)\union(\ab_2+H)$. 

We first consider the case $\pb \in H_1$. Then clearly, $\pb\notin H_2^*$. 
If we assume, on the contrary, that $\pb \in H_1^*$, then $\pb= r_1 \ab_1+r_2 \bb$ with  $r_1, r_2 \in (0,1)$. We decompose $\bb= \alpha_1\ab_1+\alpha_2\ab_2$ with $\alpha_1, \alpha_2 \in (0,1]$. This gives $\pb=(r_1+r_2\alpha_1)\ab_1 +r_2\alpha_2 \ab_2$. Since  $r_2\alpha_2 <1$ and $r_2 \alpha_2<\alpha_2$ we infer that $\pb \notin (\ab_2+H)\union (\bb+H)$. Thus $\pb \in \ab_1+H$ and $r_1 \geq 1$, a contradiction.
Consequently, $\pb\notin H_1^*\union H_2^*$.
 
A similar argument works for the case $\pb\in H_2$.
\end{proof}

\begin{Lemma}
\label{innerbounds}
 Let $\pb=(k,r)\in H_1^*$ and $\qb=(\ell,s) \in H_2^*$.
If $\bb=(u,v)$ is the bottom element of $H$, then
\begin{enumerate}
\item[(a)] $u<k<x_1$ and $v\leq r \leq y_1$.
\item[(b)]  $u\leq \ell \leq x_2$ and $v<s<y_2$.
\end{enumerate}
\end{Lemma}
 
\begin{proof}
We only show (a), part (b) is proved similarly. Clearly, $\bb\neq \pb \in \omega_H$, thus $0<u\leq k$ and $0<v\leq r$.
If $u=k$, then since $\pb\neq \bb$, we have  $v<r$. Then $r/k >v/u>y_1/x_1$, which gives that $\pb\notin H_1$, which is false. Thus $u<k$.

On the other hand, by Lemma~\ref{b-in-omega} applied in $H_1\in \Hc_2$  for $\pb$, the point 
$$
(u,v)+(x_1,y_1)-(k,r) = (u+x_1-k, v+y_1-r) \in H_1^*,
$$
hence $u<u+x_1-k$ and $v\leq v+y_1-r$. That gives $k<x_1$ and $r\leq y_1$.
\end{proof}

The following result restricts the verification of the bottom element being Ulrich to verifying that the sum of any two points in $H_i^*$ is not in $H_i^*$, for $i=1,2$.

\begin{Lemma}
\label{AG2}
Assume $\bb$ is  the bottom element of $H$. The following conditions are equivalent:
 \begin{enumerate}
	\item[(a)] $\bb$ is an Ulrich element in $H$.
	\item[(b)] For $i=1,2$, if $\pb, \qb \in H_i^*$ then $\pb+\qb \notin H_i^*$.
 \end{enumerate}
 \end{Lemma}

\begin{proof}
We know that $\bb \in G(\omega_H)$ since it is the bottom element in $H$.
If $K[H]$ is a regular ring, then $\bb=\ab_1+\ab_2$. Hence statement (a) holds by Example~\ref{gore-regular}, and (b) is true since  $H_1^*=H_2^*=\emptyset$. 

Assume that $K[H]$ is not a regular ring, hence  $\bb \in B_H\setminus \{\ab_1,\ab_2\}$. 
According to Theorem~\ref{two}, $\bb$ is an Ulrich element in $H$ if and only if for all $\pb, \qb \in B_H$ one has
\begin{equation}
\label{pq}
\pb+\qb \in (\bb+H)\union (\ab_1+H) \union (\ab_2 +H).
\end{equation}
It is of course equivalent to check \eqref{pq} for all $\pb$ and $\qb$ nonzero in $H$.

If $(\pb \in H_1$ and $\qb\in H_2$) or ($\pb \in H_2$ and $\qb \in H_1$) then $\pb+\qb \in \bb+H$, by Lemma~\ref{H12b}. 
Thus, for (a) it suffices  to check \eqref{pq} for nonzero $\pb, \qb$ both in $H_1$ or both in $H_2$.
For $i=1,2$, the semigroup $H_i$ is normal and simplicial, hence  any $\pb \in H_i$ is of the form $\pb=n_1 \bb+ n_2 \ab_i+\pb'$ with $n_1, n_2 \in \NN$ and $\pb'\in H_i^* \cup \{0\}$.
Consequently, $\bb$ is an Ulrich element in $H$ if and only if property (b) holds.
\end{proof}

\begin{Definition}
{\em
We say that {\em  $H$ is AG1} if condition (b)  in  Lemma~\ref{AG2} is satisfied for $i=1$, and we call it {\em AG2} if the said condition is satisfied for $i=2$. 
}
\end{Definition}

Thus the bottom element is an Ulrich element in $H$ if and only if $H$ is AG1 and AG2.

\begin{Remark}
\label{rem:ag1ag2}{\em
Using Lemma~\ref{innerbounds}, property AG1 means that  for any $\pb=(k,r)$ and $\qb=(\ell,s)\in H_1^*$ such that $k+\ell <x_1$
and $r+s \leq y_1$ one has that $\pb+\qb\notin H_1^*$.

 Similarly, the AG2 condition means that when $\pb=(k,r)$ and $\qb=(\ell, s) \in H_2^*$ such that $k+\ell \leq x_2$ and  $r+s<y_2$, then $\pb+\qb\notin H_2^*$.
}
\end{Remark}

\begin{Remark} 
\label{ag-zero}{\em Assume $\bb=(u,v)$ is the bottom element in $H$.
If $y_1=0$ then $v=1$, since otherwise the inequalities $v/u>(v-1)/u>y_1/x_1=0$ would give that $(u,v-1) \in \omega_{H_1}$, a contradiction to the fact that $(u,v)$ is the bottom element in $H$. Then, by Lemma~\ref{innerbounds} we get that $H_1^*=\emptyset$, hence $H$ satisfies condition AG1.

Similarly, if $x_2=0$ then $u=1$ and $H_2^*=\emptyset$; hence $H$ is AG2.
}
\end{Remark}

In order to check the AG1 and AG2 conditions we need to have a  better understanding of the points in $H_1^*$ and $H_2^*$.
We can count their elements.

\begin{Lemma}
\label{number}
Let $\bb=(u,v)$ be the bottom element for $H$. 
Then
\begin{enumerate}
\item [(a)] $|H_1^*|= vx_1-uy_1-1$ and $|H_2^*|=uy_2-vx_2-1$.
\item [(b)]  $1\leq vx_1-u y_1 \leq x_1$ and $1\leq uy_2-vx_2-1\leq y_2$.
\item [(c)] if $H_1^*\neq \emptyset$ then $vx_1-u y_1 \leq x_1-u$.
\item [(d)] if $H_2^*\neq \emptyset$ then $uy_2-vx_2 \leq y_2-v$.
\end{enumerate}
\end{Lemma}

\begin{proof} We only show the first part of (a) and (b), since the second part is proved similarly.

(a): The area of the parallelogram spanned by  $\bb$ and $\ab_1$ equals 
$\det \begin{pmatrix}x_1 & u \\ y_1 & v\end{pmatrix}=vx_1-uy_1$. Since the boundary of that parallelogram contains precisely four lattice points,  the vertices, (here we use the fact that $\gcd(u,v)=\gcd(x_1, y_1)=1)$, Pick's theorem (\cite[Theorem 2.8]{BeRo}) implies that $P_{H_1}$ has $vx_1-uy_1-1$ inner lattice points, which proves the claim.

 (b): The inequality $1\leq vx_1-uy_1$ follows from (a). Since $(u,v)$ is the bottom element of $H$, it follows that  $(u, v-1)$ is not in $\omega_H$ and in $\relint P_{H_1}$.  As $(v-1)/u <v/u$,  and $y_1/x_1 <v/u$ by our assumption, we get that $(v-1)/u \leq y_1/x_1$, i.e. $vx_1-u y_1 \leq x_1$. 

Parts (c) and (d) will be proved after Remark~\ref{k-point-bis}.
\end{proof} 

One nice consequence of Lemma~\ref{number} is a Gorenstein criterion for $K[H]$ in terms of the coordinates of the bottom element in $H$.

\begin{Corollary}
	\label{gor}
	If $\bb=(u,v)$ is the bottom element in $H$, then the $K$-algebra $K[H]$ is Gorenstein if and only if 
$vx_1-u y_1=uy_2-v x_2=1$.
 \end{Corollary}

\begin{proof}
 The ring $K[H]$ is Gorenstein if and only if $\omega_H$ is a principal ideal, hence generated by $\bb$, which is equivalent to saying that  $P_{H_1}$ and $P_{H_2}$ have no inner points. By Lemma~\ref{number} this is the case if and only if 
$vx_1-u y_1=uy_2-v x_2=1$.
\end{proof}

\begin{Lemma}
\label{k-point} 
Let $\bb=(u,v)$ be the bottom element of $H$.  We assume that $H_1^*$ is not the empty set.
 For any integer $i$ we consider the integers $q_i, r_i$ such that $iy_1 =q_i x_1+r_i$ with $0\leq r_i <x_1$.

Assume the integer $k$ satisfies $u<k<x_1$. The following statements are equivalent:
\begin{enumerate}
\item [(i)]$k$ is the $x$-coordinate of some $\pb \in H_1^*$;
\item [(ii)]$\lceil ky_1/x_1 \rceil \leq v+ \lfloor (k-u)y_1/x_1 \rfloor$;
\item[(iii)]$\lceil ky_1/x_1 \rceil = v+ \lfloor (k-u)y_1/x_1 \rfloor$;
\item [(iv)]$q_k \leq v-1 +q_{k-u}$;
\item [(v)]$q_k = v-1 +q_{k-u}$;
\item [(vi)] $r_k \geq r_{k-u}+x_1-(vx_1-uy_1)$;
\item [(vii)] $r_k = r_{k-u}+x_1-(vx_1-uy_1)$;
\item [(viii)]$r_k \geq x_1- (vx_1-uy_1)$.
\end{enumerate}
If any of these conditions holds, then  $\pb=(k,\lceil ky_1/x_1\rceil)=(k, q_k+1)$.
\end{Lemma}

\begin{proof}
Since $H_1^*\neq \emptyset$ we have that $y_1>0$ and $u<x_1$ by Remark~\ref{ag-zero} and Lemma~\ref{innerbounds}, respectively. 
 We note that for any integer $u<k<x_1$, the fractions $ky_1/x_1$ and $v+(k-u)y_1/x_1$ are not integers.
Thus  $\lceil ky_1/x_1 \rceil =q_k+1$ and $\lfloor v+(k-u)y_1/x_1 \rfloor=v+q_{k-u}$.  This shows that (ii) \iff (iv) and   
(iii) \iff (v).
Since  $q_k=(ky_1-r_k)/x_1$, simple manipulations give that (iv) \iff (vi) and (v)\iff (vii).

We also infer that the number of  points in $H_1^*$ whose $x$-coordinate is $k$ equals the number of lattice points on the line  $x=k$ located strictly between the lines $y=\frac{y_1}{x_1} x$ and $y= \frac{y_1}{x_1}(x-u)+v$, which is
\begin{eqnarray}
\label{qk}
\left\lfloor \frac{y_1}{x_1}(k-u)+v \right\rfloor - \left\lceil \frac{y_1}{x_1} k \right\rceil +1  &=&
v+q_k+\left\lfloor \frac{r_k- uy_1}{x_1}\right\rfloor-(q_k +1)+1  \\ 
\label{rk}
&=&  \left\lfloor \frac{v x_1-uy_1 +r_k}{x_1}\right\rfloor \in \{0,1\}.
\end{eqnarray}
The latter statement is due to the fact that $r_k<x_1$ and $vx_1-u y_1 \leq x_1$, by Lemma~\ref{number}.

Consequently, $k\in (u, x_1)$ is the $x$-coordinate of some point in $H_1^*$ if and only if the value in equation \eqref{qk} is at least (and actually equal to) $1$, which is equivalent to property (ii), respectively to (iii). That is, moreover, equivalent (using \eqref{rk}) to 
$$1\leq \frac{vx_1-uy_1+r_k}{x_1},$$
which can be rewritten as $r_k \geq x_1-(vx_1-uy_1)$, namely statement (viii).

From \eqref{qk} and \eqref{rk} we obtain that if $k$ is the $x$-coordinate of some point $\pb \in H_1^*$, then $\pb=(k, \lceil{\frac{y_1}{x_1}k} \rceil ) = (k, q_k+1)$.
\end{proof}

\begin{Remark} 
\label{k-point-bis} {\em
A similar result holds for the points in $H_2^*$  in terms of the integers $q'_i, r'_i$ such that $i x_2 =q'_i y_2+r'_i$ with $0\leq r'_i <y_2$. }
\end{Remark}

Now we can finish the proof of Lemma \ref{number}.

\begin{proof} (of Lemma \ref{number},  continued). 

(c): By Lemma~\ref{k-point}, for each  $u<k<x_1$ there is at most one point in $H_1^*$ whose $x$-coordinate is $k$, therefore 
$|H_1^*|\leq  x_1-u-1$.  Using point (a) we obtain the inequality at (c). Part (d) is proved similarly.
\end{proof}

It will be convenient to denote $\pi_1(H_1^*)=\{k: \text{ there exists }(k,\ell) \in H_1^* \}$.
The next result is a criterion to verify the AG1 property in terms of the remainders $r_i$ introduced in Lemma~\ref{k-point}
, with  $i\in \pi_1(H_1^*)$. A similar statement characterizes the AG2 property in terms of the $r'_j$'s from Remark~\ref{k-point-bis}, with $j\in \pi_2(H_2^*)$.

\begin{Proposition}
\label{AG1-rk}
For any integer $i$ let $r_i  \equiv  iy_1  \mod x_1$ with $0\leq r_i <x_1$.
Then $H$ is AG1 if and only if $r_k+ r_\ell  < 2x_1 -(vx_1-uy_1)$ for all integers $k,\ell \in \pi_1(H_1^*)$  with $k+\ell<x_1$.
\end{Proposition}


\begin{proof}
If $H_1^*=\emptyset$ then there is nothing to prove. Assume $H_1^*$ is not empty. 
If $k, \ell \in \pi_1(H_1^*)$ then by Lemma \ref{k-point}, $\pb_1=(k, \lceil k y_1/x_1\rceil)$ and $\pb_2=(\ell, \lceil \ell y_1/x_1\rceil)$ are the corresponding points in $H_1^*$.
By definition, $H$ is AG1 if and only if $\pb_1+\pb_2 \notin H_1^*$ for all $\pb_1$ and $\pb_2$ as above. 
When $k+\ell \geq x_1$,  Lemma \ref{innerbounds} implies already that  $\pb_1+\pb_2 \notin H_1^*$.
If  $k+\ell < x_1$, then $\pb_1+\pb_2 \notin H_1^*$ if and only if 
\begin{eqnarray}
\label{rightnot1} \left\lceil \frac{ky_1}{x_1} \right\rceil + \left\lceil\frac{\ell y_1}{x_1} \right\rceil  &\geq&  (k+\ell -u) \frac{y_1}{x_1} +v, \quad \text{ equivalently} \\ 
\nonumber  \frac{ky_1- r_k}{x_1}+1 +\frac{\ell y_1-r_\ell}{x_1}+1 &\geq&  (k+\ell -u) \frac{y_1}{x_1} +v, \\
\nonumber k y_1 -r_k + \ell y_1-r_\ell +2 x_1 &\geq&  (k+\ell)y_1-uy_1+vx_1,\\
\label{rightnot2} 2x_1-(vx_1-u y_1) &\geq&  r_k+ r_\ell.
\end{eqnarray} 
Since $u<k+\ell<x_1$, the term of  the right hand side of \eqref{rightnot1} is not an integer, hence the inequality at \eqref{rightnot1} (and equivalently, at \eqref{rightnot2}) can not become an equality.
\end{proof}

\section{A criterion for $(1,1)$ to be an Ulrich element}

Our aim in this  section is to obtain a complete classification of when  $\bb=(1,1)$ is an Ulrich element.
The setup in Notation~\ref{not} is in use.
The element $(1,1)$ is in $\omega_H$ if and only if $y_1/x_1<1< y_2/x_2$.  
If that is the case, it is clear that $(1,1)$ is the bottom 
element in $H$.    
It suffices to verify the  AG1 and AG2 conditions, by Lemma~\ref{AG2}.
 
 Set $n=x_1-y_1-1$, which equals $|H_1^*|$, by Lemma~\ref{number}. If $n=0$, then $H$ is clearly AG1.

We consider the case $n>0$. 
The next result presents an explicit way to determine $H_1^*$.
Recursively, we define non-negative integers $\ell_1,\ldots, \ell_n$ and $s_1,\ldots,s_n$ by
\[
x_1=\ell_1(x_1-y_1)+s_1,   \quad \text{with} \quad s_1<x_1-y_1,
\]
and
\[
y_1+s_{i-1}=\ell_{i}(x_1-y_1)+s_{i}\quad \text{with} \quad  s_i<x_1-y_1,
\]
for $i=2,\ldots,n$.

\begin{Lemma}\label{elements-H1}
	Assume that $(1,1)$ belongs to $\omega_H$  and $H_1^*\neq \emptyset$.   Then
	\[
	H_1^*=\left\{\pb_t=(c_t,d_t) : c_t=t+\sum^t_{i=1}\ell_i \ , \ d_t=\sum^t_{i=1}\ell_i \ , \ t=1,\dots,n\right\},
	\]
	\end{Lemma}

\begin{proof}
	For $k=1,\ldots, x_1-1$, let $ky_1=q_kx_1+r_k$ with integers $q_k\geq 0$ and $x_1>r_k\geq 0$. 

           By Lemma~\ref{k-point}, 	 the integer $k>1$ is the $x$-coordinate of an element  of $H_1^*$ if and only if $q_k=q_{k-1}$.
In this case, $(k,1+q_k)\in H_1^*$.

	Now, let $t\geq 1$. Summing up the equations  $x_1= \ell_1(x_1-y_1)+s_1$ and $y_1+s_{i-1}=\ell_{i}(x_1-y_1)+s_{i}$,
	for $i=2,\dots,t$, we get
	\[
	x_1+(t-1)y_1+s_1+s_2+\dots+s_{t-1}=\sum^t_{i=1}\ell_i(x_1-y_1)+s_1+s_2+\dots+s_t,
	\]
	consequently,
	\[
	\left(t-1+\sum^t_{i=1}\ell_i \right) y_1=\left(\sum^t_{i=1}\ell_i-1 \right)x_1+s_t.
	\]
	Then
	\[
	\left(t+\sum^t_{i=1}\ell_i\right) y_1= \left(\sum^t_{i=1}\ell_i-1 \right)x_1+s_t+y_1,
	\]
	with  $s_t+y_1<x_1$. Therefore, $q_k=q_{k-1}=(\sum^t_{i=1}\ell_i-1)$, for $k=t+\sum^t_{i=1}\ell_i$. 

Note that 
\begin{eqnarray*}
n+ \sum_{i=1}^n \ell_i  &=& n+\frac{x_1-s_1}{x_1-y_1}+\sum_{i=2}^t \frac{y_1-s_{i-1}+s_i}{x_1-y_1} \\
   &= & n+ \frac{x_1+(n-1)y_1-s_n}{x_1-y_1}=n+1+ \frac{n {y_1}-s_n}{x_1-y_1} \\
&<& n+1+ y_1= x_1,     
\end{eqnarray*}
hence $\pb_t=(t+\sum^t_{i=1}\ell_i,\sum^t_{i=1}\ell_i)\in H_1^*$ for $t=1,\dots, n$.
	
	We know from Lemma~\ref{number}, that $H_1^*$ has exactly $n=x_1-y_1-1$ elements, so $\pb_1,\dots,\pb_n$ are the only elements of $H_1^*$.
\end{proof}

\begin{Examples}\label{Ex}{\em
	Let $x_1=\ell_1(x_1-y_1)+s_1$ and $y_1+s_{i-1}=\ell_{i}(x_1-y_1)+s_{i}$,
	for $i=2,\dots,n=x_1-y_1-1$ as before.
	\begin{enumerate}
		\item[(a)] If $y_1=1$, then $H_1^*=\{(m,1) \ : \ m=2,\dots,x_1-1\}$. In this case, $H$ is AG1 by Lemma~\ref{AG2}.
		\item[(b)] If $x_1-y_1 \in \{1,2\}$ then by Lemma~\ref{number}, $H_1^*$ is either empty, or it consists of one element, which is different from $(0,0)$. Hence $H$ is AG1. 
		\item[(c)]  If $2<2y_1<x_1<3y_1$, then $\ell_1=\ell_2=1$.
		Therefore, $\pb_1=(2,1)$ and $\pb_2=(4,2)=2\pb_1$ belong to $H_1^*$. Then, by definition,  $H$ is not AG1.
	\end{enumerate}
}
\end{Examples}

The next theorem   gives a simple arithmetic criterion to check the AG1 or AG2 property.
 
\begin{Theorem}
\label{ulrich11}
	Assume that $(1,1)$ belongs to $\omega_H$. Assuming Notation~\ref{not}, then  
\begin{enumerate} 
\item [(a)] $H$ is AG1  if and only if  $x_1 \equiv 1 \mod (x_1-y_1)$; 
\item [(b)] $H$ is AG2  if and only if  $y_2  \equiv 1 \mod (y_2-x_2) $; 
\item [(c)] $(1,1)$ is an Ulrich element in $H$ if and only if   $x_i \equiv 1 \mod (x_i-y_i)$ for $i=1,2$.
\end{enumerate}
\end{Theorem}

\begin{proof}
(a): Let $n=x_1-y_1-1= |H_1^*|$. If $n\in \{0,1\}$, then $H$ is AG1 by Example~\ref{Ex}(b). On the other hand, if $n=0$ then $x_1-y_1=1$ and clearly, $x_1\equiv 1 \mod (x_1-y_1)$.
When $n=1$ we have $x_1-y_1=2$. Since $\gcd(x_1,y_1)=1$ we get  that $x_1$ is odd, hence $x_1\equiv 1 \mod (x_1-y_1)$, too.

We further prove the stated equivalence when $n\geq 2$.
Let  $\ell_1,\dots,\ell_n\geq0$ and  $x_1-y_1> s_1,\dots,s_n\geq0$ such that
	\[
	x_1=\ell_1(x_1-y_1)+s_1 \ , \ y_1+s_{i-1}=\ell_{i}(x_1-y_1)+s_{i},
	\]
	for $i=2,\dots,n$. Then
	\[
	H_1^*=\{\pb_t=(c_t,d_t) : c_t=t+\sum^t_{i=1}\ell_i \ , \ d_t=\sum^t_{i=1}\ell_i \ , \ t=1,\dots,n\},
	\]
	by Lemma~\ref{elements-H1}. 
We note that since $y_1>0$ (see Remark~\ref{ag-zero}) we have $x_1>x_1-y_1$, hence $\ell_1 \geq 1$. 

Assume that   $x_1 \equiv 1 \mod (x_1-y_1)$. Then it is easy to check that $s_i=i$ and $\ell_i=\ell_1-1$ for $i=2,\dots, n$. 
  Consequently,
	\[
	H_1^*=\{(t\ell_1+1,t(\ell_1-1)+1) : t=1,\dots,n\},
	\]
	and therefore,  the sum of any two elements of $H_1^*$ is not in $H_1^*$, i.e. $H$ is AG1.

	Conversely, assume that $H$ is AG1.
As $n>0$ we get that $x_1-y_1 >1$ and $y_1 >0$. In case $y_1=1$, then clearly, $x_1\equiv 1 \mod (x_1-y_1)$.

We consider the case $y_1\geq 2$.
As $1=\gcd(x_1, y_1)=\gcd(x_1, x_1-y_1)=\gcd(s_1, x_1-y_1)$ and $x_1-y_1 >1$ we have that $s_1>0$. We need to prove that $s_1=1$. 

Assume, on the contrary, that $s_1 \neq 1$. Then $s_1\geq 2$. Since 
\begin{eqnarray*}
(\ell_1-1)(x_1-y_1)+s_1=y_1 &\leq & y_1+s_{i-1}= \ell_i(x_1-y_1)+s_i \text{ \quad and} \\
y_1+s_{i-1} &<&  y_1+ (x_1-y_1)=x_1= \ell_1(x_1-y_1)+s_1,
\end{eqnarray*}
we have $\ell_1-1 \leq \ell_i \leq \ell_1$, for $i=2,\dots, n$.

If $\ell_2=\ell_1$, then $\pb_2=(2+2 \ell_1,2\ell_1)=2\pb_1$ which contradicts the AG1 property.
Now, we consider the case $\ell_2=\ell_1-1$. By subtracting the equations
\begin{eqnarray*}
x_1= \ell_1 (x_1-y_1)+s_1  \quad  \text{ and }  \quad y_1+s_1=\ell_2(x_1-y_1)+s_2,
\end{eqnarray*}
we get that $s_2=2s_1$,  hence $s_2>s_1$. 

If $\ell_2=\dots = \ell_n$ then $s_1<s_2<\dots <s_n$ is an increasing sequence of $n$ positive integers less than $n+1$, hence $s_1=1$, which is false.
Thus $\ell_i= \ell_1$ for some $i\geq 3$. Let $i$ be the smallest index with this property, i.e. $\ell_2=\dots= \ell_{i-1}=\ell_1-1$ and $\ell_i= \ell_1$.  Then
	\begin{eqnarray*}
		\pb_i&=&\left(i+(i-2)(\ell_1-1)+2\ell_1 , (i-2)(\ell_1-1)+2\ell_1)\right)\\
		&=&(1+\ell_1, \ell_1)+\left(i-1+(i-2)(\ell_1-1)+ \ell_1,(i-2)(\ell_1-1)+ \ell_1\right)\\
                     &=&\pb_{1}+\pb_{i-1},
	\end{eqnarray*}
	a contradiction.
This shows that when $H$ is AG1, then $x_1\equiv 1 \mod (x_1-y_1)$.

For part (b) we let $H'$ be the semigroup in  $\Hc_2$ with the extremal rays $\ab'_1=(y_2, x_2)$  and $\ab_2'=(y_1, x_1)$.
We remark that  $H$ is AG2 if and only if $H'$ is AG1, and we use (a). Part (c) is a consequence of (a) and (b). 
\end{proof}

\begin{Corollary}
\label{ag11}
Let $H$ be a semigroup in $\Hc_2$ with extremal rays $\ab_i=(x_i, y_i)$ for $i=1,2$. Assume $(1,1) \in \omega_H$ and $x_1x_2y_1y_2 \neq 0$ . Then $K[H]$ is AG if and only if $x_i \equiv 1 \mod (x_i-y_i)$ for $i=1,2$. 
\end{Corollary}

\begin{proof}
By Proposition~\ref{1}, the only possible Ulrich element in $H$ is $(1,1)$. Conclusion follows by  Theorem~\ref{ulrich11}.
\end{proof}

\begin{Remark}
\label{ulrich-not-ag}{\em  In the statement of Corollary~\ref{ag11}, the assumption $x_1x_2y_1y_2\neq 0$ can not be dropped.
For instance, let $H \in \Hc_2$ with the extremal rays $\ab_1=(1,0)$ and $\ab_2=(2,5)$. Its Hilbert basis is
$B_H=\{\ab_1, \ab_2, \cb_1=(1,1), \cb_2=(2,3), \cb_3=(1,2)\}$. The bottom element in $H$ is $\cb_1$, and by Theorem~\ref{ulrich11} it follows that  $H$ is not AG2. 

Still, $H$ is AG. Since $2\cb_1=(2,2)=\cb_3+\ab_1$, $2\cb_2=(4,6)=\ab_1+\ab_2+\cb_1$  and $\cb_1+\cb_2=(3,4)=\ab_1+2\cb_3$, by Theorem~\ref{two} we get that $\cb_3$ is an Ulrich element in $H$.}
\end{Remark}

\section{Nearly Gorenstein semigroup rings}
\label{sec:ng}

In this section we prove the nearly Gorenstein property for semigroup rings $K[H]$ when $H\in \Hc_2$.

Nearly Gorenstein rings approximate Gorenstein rings in a  different way as  almost Gorenstein rings.  In \cite{HHS},  a local (or graded) Cohen--Macaulay ring which admits a canonical module $\omega_R$  is called {\em nearly Gorenstein}  if the  trace  of $\omega_R$ contains the (graded) maximal ideal of $R$. In the case that $R$ is a domain, the canonical module  can be realized as an ideal of $R$ and its trace in $R$, which we denote by  $\tr(\omega_R)$,  is the ideal $\sum_ff\omega_R$,  where the sum is taken over all $f$ in the quotient field of $R$ for which $f\omega_R\subseteq R$, see \cite[Lemma 1.1]{HHS}.

A one-dimensional almost Gorenstein  ring is nearly Gorenstein, but the converse does not hold in general. In higher dimension there is in general no implication valid  between these two concepts, see \cite{HHS}.

\begin{Theorem}
	\label{nearly}
	Let $H$ be a simplicial affine semigroup
	in $\Hc_2$. Then $R=K[H]$ is a nearly Gorenstein ring.
\end{Theorem}

\begin{proof}
	Let $\ab_1=(c,d)$ and $\ab_2=(e,f)$ be the extremal rays of $H$. We may assume that $d/c <f/e$ and 
that $R$ is not already a Gorenstein ring.
	
	The vector $\nb_1=(-d,c)$ is orthogonal to $\ab_1$ and $\nb_2=(f,-e)$ is orthogonal to $\ab_2$. Moreover,  $\cb$ is in $C$, the cone over $H$,   if and only if $\langle \nb_1,\cb\rangle \geq 0$ and  $\langle \nb_2,\cb\rangle \geq0$.
	
	Let $\cb_1,\ldots, \cb_t, \cb_{t+1}, \cb_{t+2}$ be the Hilbert basis of $H$, where $\cb_{t+i}=\ab_i$ for $i=1,2$.  Then $\omega_R$ is generated by $v_i=\xb^{\cb_i}$ for $i=1,\ldots, t$, see Lemma~\ref{lemma:omega}.
	
	In order to prove that $R$ is nearly Gorenstein, it suffices to show that for each element $\cb_i$  of the Hilbert basis there exist $\cb\in \ZZ^2$ and an integer $k\in \{1,\ldots,t\}$ such that
	\begin{enumerate}
		\item[(i)] $\cb+\cb_j\in C$ for $j=1,\ldots,t$, and
		\item[(ii)] $\cb+\cb_k=\cb_i$.
	\end{enumerate}
	If $i\in\{1,\ldots,t\}$, we may choose $\cb=0$ and $k=i$. It suffices to consider the cases $i=t+1$  and $i=t+2$. By symmetry we may assume that $i=t+1$, and have  to find $\cb\in\ZZ^2$ and  $k\in\{1,\ldots,t\}$ such that (i) is satisfied and such that $\cb+\cb_k=\ab_1$.
	
	Let $k\in\{1,\ldots,t\}$ be   chosen such that 
$\langle \nb_1,\cb_k\rangle =\min \{\langle \nb_1,\cb_j\rangle: j=1,\ldots,t \}  $.
	 Set  $\cb=\ab_1-\cb_k$.  Then $\cb+\cb_k=\ab_1$.
	Moreover, by the choice of $k$ for $j=1,\ldots,t$  we have
	\begin{eqnarray*}
		\langle \nb_1,\cb+\cb_j\rangle= \langle \nb_1,\ab_1\rangle -\langle \nb_1,\cb_k\rangle +\langle \nb_1,\cb_j\rangle=
		0-\langle \nb_1,\cb_k\rangle +\langle \nb_1,\cb_j\rangle\geq 0,
	\end{eqnarray*}
	and
	\[
	\langle \nb_2,\cb+\cb_j\rangle = \langle \nb_2,\ab_1\rangle -\langle \nb_2,\cb_k\rangle +\langle \nb_2,\cb_j\rangle.
	\]
	Since $\cb_j\in H$, we have $\langle \nb_2,\cb_j\rangle\geq 0$. Let $L$ be the line  passing through $\cb_k$  which is  parallel to $L_2=\RR \ab_2$, and  $L'$ be the line   passing  through $\ab_1$ parallel to $L_2$. Since $\cb_k\in P_H$, the line $L$ has smaller distance to $L_2$ than the line $L'$. This implies that $\langle \nb_2,\ab_1\rangle >\langle \nb_2,\cb_k\rangle$, hence $\langle \nb_2,\cb+\cb_j\rangle>0$. Thus we conclude that $\cb+\cb_j\in C$, as desired.
\end{proof}

Theorem~\ref{nearly} is no longer valid when $\dim K[H] >2$,  as the following example shows.

\begin{Example}\label{example2}
	{\em
	We consider again the semigroup $H\in\Hc_3$ from Remark~\ref{rem:nobottom}.
  It turns out that $K[H]$ is not nearly Gorenstein for this semigroup $H$. One can see that $\ab_1$ does not satisfy the two conditions (i) and (ii) in the proof of Theorem~\ref{nearly}.  

In fact, if we consider the set  $A$ of all  $\ab_1-\cb_i$ for $i=1,\dots,13$, then the third component of elements in $A$, belongs to $\{0,-1,-2,-3,-4\}$. Adding the elements with negative third component to $(1,2,1)$, we get a vector with third component less than $1$, which does   not belong to  $C$, the cone over $H$. 
Adding those elements in $A$ with zero third component to either $(2,1,1)$ or $(1,2,1)$, we again get a vector which does not belong to  $C$.
}
\end{Example}

\medskip
{\bf Acknowledgement}.
We gratefully acknowledge the use of the  Singular (\cite{Sing}) and Normaliz (\cite{Normaliz}) softwares for our computations.
Raheleh Jafari was in part supported by a grant from IPM (No. 98130112).
 Dumitru I. Stamate  was partly supported by the University of Bucharest, Faculty of Mathematics and
Computer Science through the 2018 and 2019 Mobility Fund.
\medskip

\medskip


\end{document}